\pdfoutput=1
\documentclass[11pt,reqno]{amsart}

\usepackage[T1]{fontenc}
\usepackage[utf8]{inputenc}
\usepackage[a4paper,margin=2.5cm]{geometry}
\usepackage{amssymb}
\usepackage{mathtools}
\allowdisplaybreaks

\usepackage{graphicx}
\usepackage{xcolor}
\usepackage{enumitem}
\usepackage{float}

\usepackage{multirow}
\usepackage{booktabs}
\usepackage{adjustbox}

\usepackage{hyperref}

\def\doi#1{\href{https://doi.org/\detokenize{#1}}{\url{https://doi.org/#1}}}

\usepackage[ruled,vlined,linesnumbered]{algorithm2e}
\SetEndCharOfAlgoLine{}
\DecMargin{2.2mm}
\SetKwInput{Input}{Input}\SetKwInput{Output}{Output}
\SetKwFor{While}{While}{do}{}
\SetKwFor{ForEach}{For each}{do}{}
\SetKwIF{If}{ElseIf}{Else}{If}{then}{else if}{else}{endif}
\SetKw{Return}{Return}
\SetKw{KwRet}{return}
\SetKw{Recurse}{Recurse}
\SetKwFor{For}{For}{do}{}
\SetKw{Break}{break}
\SetKw{KwAnd}{and}
\SetKwFunction{Simon}{IsotropicVector}
\SetKwFunction{Verify}{Verify}
\SetKwFunction{LLL}{LLL}
\SetKwFunction{Enumerate}{Enumerate}
\SetKwFunction{ShortestVector}{ShortestVector}
\SetKwFunction{Cornacchia}{Cornacchia}

\theoremstyle{plain}
\newtheorem{theorem}{Theorem}[section]
\newtheorem{lemma}[theorem]{Lemma}
\newtheorem{proposition}[theorem]{Proposition}
\newtheorem{corollary}[theorem]{Corollary}
\theoremstyle{definition}
\newtheorem{definition}[theorem]{Definition}
\theoremstyle{remark}
\newtheorem{remark}[theorem]{Remark}

\newcommand{\A}{\mathcal{A}}

\newcommand{\C}{\mathcal{C}}
\newcommand{\Q}{\mathbb{Q}}
\newcommand{\F}{\mathbb{F}}
\newcommand{\Z}{\mathbb{Z}}

\newcommand{\Bp}{B_p}
\newcommand{\mi}{\mathbf{i}}
\newcommand{\mj}{\mathbf{j}}

\newcommand{\Ord}{\mathcal{O}}
\newcommand{\rhs}{\mathrm{rhs}}
\newcommand{\theHmat}{\begin{psmallmatrix} u & \alpha \\ \bar\alpha & v \end{psmallmatrix}}

\newcommand{\legendre}[2]{\ensuremath{\left( \frac{#1}{#2} \right) }}

\newcommand{\lin}[1]{line~\ref{#1}}
\newcommand{\Lin}[1]{Line~\ref{#1}}
\newcommand{\lins}[2]{lines~\ref{#1}--\ref{#2}}
\newcommand{\Lins}[2]{Lines~\ref{#1}--\ref{#2}}

\DeclareMathOperator{\GL}{GL}

\DeclareMathOperator{\Gen}{Gen}

\DeclareMathOperator{\cl}{cl}

\DeclareMathOperator{\tr}{tr}
\DeclareMathOperator{\Div}{Div}
\DeclareMathOperator{\Pic}{Pic}
\DeclareMathOperator{\Aut}{Aut}
\DeclareMathOperator{\NS}{NS}
\DeclareMathOperator{\rank}{rank}
\DeclareMathOperator{\Hom}{Hom}
\DeclareMathOperator{\End}{End}
\DeclareMathOperator{\charec}{char}

\DeclareMathOperator{\disc}{disc}

\newcommand{\curlyO}{\mathcal{O}}

\DeclareMathOperator{\nrd}{nrd}

\begin{document}

\title{Detecting split surfaces, RM surfaces, and minimum walks on isogeny graphs}

\author{Eda K\i{}r\i{}ml\i{}}
\address{University of Birmingham, UK}
\email{e.kirimli@bham.ac.uk}

\author{Gaurish Korpal}
\address{University of Auckland, New Zealand}
\email{gaurish.korpal@auckland.ac.nz}


\begin{abstract}
We study the detectability of split surfaces in isogeny graphs of principally polarized superspecial abelian surfaces, a question relevant to the security analysis of dimension-$2$ isogeny-based cryptography. Our approach uses refined Humbert invariants to replace explicit isogeny computations with primitive representation problems for positive definite quadratic forms in five variables. We develop algorithms to detect $(N,N)$-splittings and to compute the minimum $(N,N)$-splitting level of a superspecial Jacobian without constructing the corresponding isogeny path. The same framework detects embeddings of real multiplication (RM) orders through primitive representations of discriminants of real quadratic orders. For RM, we use exhaustive small-prime data together with $100{,}000$ random polarizations per prime for $227\leq p\leq1619$, testing primitive representations of square-free discriminants $D\leq100$; the first nontrivial primitively represented discriminant is typically small. 
We apply these methods experimentally in two regimes. For $11\leq p\leq251$, where the irreducible principal polarizations are known exhaustively, the automorphism data recovered from the refined Humbert invariant reproduces the counts of Ibukiyama--Katsura--Oort for the irreducible polarizations. For large parameters, we reach primes of $1000$ bits, where the splitting degrees produced satisfy $\log_2N\approx\log_2p$ with a distribution whose shape is stable across the whole range, so the detected $(N,N)$-isogeny has degree about $p^{2}$; the smallest observed largest prime-power divisor of $N$ remains small up to $250$ bits and increases sharply from $300$ bits onward.
\end{abstract}

\maketitle

\section{Introduction}

Isogeny-based cryptography is an active area of post-quantum cryptography, concerned with variants of the isogeny problem, which are believed to be computationally hard. Following the attacks on SIDH \cite{sidhattack,sidhattack2,sidhattack3}, increasing attention has been directed toward higher-dimensional settings, particularly principally polarized superspecial abelian surfaces. A better understanding of the structure and expansion properties of the isogeny graphs of principally polarized superspecial abelian surfaces is important for assessing the hardness of the general dimension-$2$ isogeny problem.

In this work, we focus on the splitting behavior of principally polarized superspecial abelian surfaces, which plays an important role in the security analysis of several constructions and is also a natural geometric problem in its own right. In particular, contrary to the increase in the number of cryptographic primitives based on principally polarized superspecial abelian surfaces, the structure of the associated isogeny graphs and the security assumptions derived from them remain only partially understood.

We call a \emph{product surface}\footnote{We follow Kani's terminology here.} shortly the product of two elliptic curves, and call the product surface with product polarization \emph{a split surface}\footnote{ This is mostly referred to as a principally polarized superspecial split surface in the cryptography papers.}. A curve $\C$ of genus $2$  is said to be \emph{splitting} if its Jacobian is $K$-isogenous to a principally polarized abelian surface $E_1\times E_2$ where $E_i/K$ for $i=1,2$ are supersingular elliptic curves over a finite field $K$.

The main focus of this work is the detection of split surfaces in the isogeny graph of principally polarized superspecial abelian surfaces and the embeddings of RM orders in the endomorphism rings. The splitting behavior of isogeny graphs is important to understand the difficulty of isogeny problems in dimension~$2$. Recently, isogeny graphs of RM surfaces gained attention due to their Ramanujan properties shown by \cite{CharlesGorenLauter2009FamiliesRamanujan}.
We give an algorithm to detect split surfaces encountered at random in the isogeny graph and use it to study shortest paths from Jacobians to split surfaces. Moreover, our methods allow us to show all possible embeddings of RM orders when only the principal polarization is provided as input.

We approach these problems using the arithmetically rich invariants of principally polarized abelian surfaces known as \emph{refined Humbert invariants}, introduced by Kani \cite{Kani94} in 1994. The main advantage is that they enable the relevant computations without explicitly evaluating costly chains of dimension-$2$ isogenies. 

Of particular relevance is the Costello--Smith algorithm~\cite{CraigBen} for the superspecial isogeny problem, whose first phase walks from a Jacobian to a split surface and therefore depends directly on how split surfaces are distributed; we describe it in Section~\ref{applications}.

While it is well known that split surfaces form only a small proportion of the full graph, comparatively little is understood about their distribution within the graph. Our results provide further insight into the security and performance of recent cryptographic constructions, while also contributing to a broader understanding of superspecial isogeny graphs in dimension~$2$.

Let $\A$ be an abelian surface over $\overline{\F}_{p}$ carrying a principal polarization $\theta$. The intersection product on the N\'eron--Severi group $\NS(\A)$ induces a positive definite integral quadratic form $q_{(\A,\theta)}$ on the quotient $\NS(\A)/\Z\theta$, called the \emph{refined Humbert invariant} of $(\A,\theta)$. Computing refined Humbert invariants is in general difficult, and has been studied extensively by Kani~\cite{KaniMJ,KaniESCII,KaniESCI,SubcoversGenus2}; when $\A\cong E_1\times E_2$ for elliptic curves $E_1,E_2$ it becomes tractable, which is the case we exploit throughout. The construction is recalled in full in Appendix~\ref{sec:RHI}; see also \cite{Kani94,KaniESCI}.

Several previous works have used refined Humbert invariants, either directly or through the associated Humbert surfaces, in the study of isogeny-based cryptography.
The computational problem of determining a refined Humbert invariant is itself studied in \cite{kirimlimartindale}, where it is shown to be polynomially equivalent, under GRH, to the computational isogeny problem of elliptic curves. The works \cite{Splitting,detectionofendomorphisms} instead use the algebraic Humbert surfaces $\mathcal{H}_{N^2}$ to study splitting and endomorphism problems. Their methods, however, rely on explicitly defining equations for $\mathcal{H}_{N^2}$ when $ \mathcal{H}_{N^2} $ is birational, and are restricted to $N\leq 11$. By working directly with refined Humbert invariants, we avoid the need for explicit equations parameterizing Humbert surfaces.

Our use of refined Humbert invariants rests on three properties, set out in Section~\ref{applications}: they decide whether a principal polarization is irreducible, they detect $(N,N)$-splittings, and they detect embeddings of real quadratic orders. In this way, we contribute to the understanding of superspecial isogeny graphs by considering problems such as their splitting behavior and the distribution of RM superspecial surfaces.

In this work, we study principally polarized superspecial abelian surfaces through their quintic refined Humbert invariants. For splitting problems, we fix a superspecial product surface $\A=E\times E$, choose a random principal polarization $\theta$, and compute $q_{(\A,\theta)}$. After testing irreducibility by Proposition~\ref{irreducibilitycriteria}, we apply Theorem~\ref{SplittingTheorem} to detect $(N, N)$-splittings through primitive representations of $N^2$. This framework is then used to study the distribution of split surfaces and shortest paths from Jacobians to split surfaces, with particular emphasis on small primes $N=\ell$, and composite $N$.
As another use, refined Humbert invariants also detect real multiplication (RM): a primitive representation of a fundamental discriminant $D$ by $q_{(\A,\theta)}$ yields an embedding of the corresponding RM order $\mathcal{O}_D$ into the endomorphism ring $\End(\A)$.

We study two complementary primitive-representation problems associated with quintic refined Humbert invariants. First, for a prescribed prime $\ell$, we determine the least exponent $k$ for which $q_{(\A,\theta)}$ primitively represents $\ell^{2k}$. By Kani's splitting theorem (Theorem~\ref{SplittingTheorem}), this detects the minimum $(\ell^k,\ell^k)$-splitting, and gives arithmetic information about the distance to split surfaces. Second, without prescribing the splitting degree, we reduce the search for a primitive square representation $q_{(\A,\theta)}(x)=N^2$ to finding an isotropic vector of an associated indefinite  quadratic form of rank six. This yields an effective method for producing a composite degree $N$, hence an existence of a split $(N, N)$-isogeny.

We investigate these methods experimentally in two regimes. For small characteristic, where the principal polarizations are known exhaustively, we study the distribution of minimum splitting levels for $\ell=2,3$, together with RM and automorphism data. For larger parameters, we compute splitting levels for a $50$-bit prime and apply the unrestricted square-witness method to primes up to $1000$ bits, examining both the size and the arithmetic structure of the resulting splitting degrees. 

\paragraph{Our contributions.}
\begin{itemize}[leftmargin=1.4em,itemsep=2pt,topsep=2pt]
\item We show that splitting detection and real-multiplication detection for principally polarized
superspecial abelian surfaces reduce to \emph{primitive-representation problems} for a
positive definite quadratic form in five variables, the quintic refined Humbert invariant. The
reduction takes the polarization data itself as input, so no Igusa invariants, no defining
equations for a genus $2$ curve and no explicit chain of isogenies are required at any stage.
\item We realize this reduction in three algorithms: Section~\ref{designing} gives a sampler for principal
polarizations that works far beyond the range of exhaustive enumeration
(Algorithm~\ref{alg:polzrandom}), a lattice-enumeration procedure for the least splitting level
$k_0$ at a prescribed $\ell$ (Algorithm~\ref{alg:kzero}), and a reduction of the unrestricted
problem ``find some splitting degree $N$'' to finding an isotropic vector of an indefinite rank-six
form, solved with Simon's algorithm (Algorithm~\ref{alg:witness}).
\item We relate $k_0$ precisely to distance in the Richelot graph (Lemma~\ref{lem:k0-distance}): it is
the length of the shortest walk to the split locus made of good extensions, and an upper bound for
the graph distance.
\item We apply these methods experimentally in two regimes: exhaustively for small $p$, where all
principal polarizations are known, and up to $1000$-bit $p$ for the unrestricted problem. We use
the resulting data to compare automorphism counts against Ibukiyama--Katsura--Oort and to
estimate how often real multiplication by a prescribed order occurs.
\end{itemize}

\paragraph{Organization} Section~\ref{sec:prelim} recalls isogeny graphs of abelian surfaces and
refined Humbert invariants; the full construction of the latter is deferred to Appendix~\ref{sec:RHI}.
Section~\ref{designing} develops the three algorithms. Section~\ref{applications} applies them to
splitting and to real multiplication, and Section~\ref{data} presents the experimental data.
Section~\ref{future} lists future directions.

We implemented all the algorithms from Section~\ref{designing} using a Julia package called \texttt{OSCAR}~\cite{OSCAR,OSCAR-book}, especially using its dependency package \texttt{Nemo/Hecke}~\cite{nemo}.
The experiments were run under Julia~1.12.7 with \texttt{OSCAR}~1.8.2. The enumeration and sampling of polarizations, the refined Humbert invariants at large $p$, and Algorithm~\ref{alg:kzero} were run on an HPC cluster; the remaining computations were run on a workstation with a $24$-core Intel i7-12800HX and $16$\,GB of memory. Wall-clock timings for the individual runs are recorded in the output files of the repository; we make no performance claims from them. The code and data are provided as part of the supplemental material and are available at:

\begin{center}
\url{https://github.com/gkorpal/humbert-split}
\end{center}

\section{Preliminaries}\label{sec:prelim}
 In this section, we introduce the terminology and recall several results from the literature concerning isogenies of superspecial principally polarized abelian surfaces. In particular, we review the theory of refined Humbert invariants, real multiplications, isogeny graphs of abelian surfaces, and some basic facts about primitive representations by integral quadratic forms.

\subsection{Superspecial abelian varieties}
An elliptic curve is an abelian variety of dimension $1$. An isogeny between elliptic curves is a surjective homomorphism with finite kernel. An isogeny from an elliptic curve to itself is called an endomorphism, and the set of endomorphisms of $E$ forms the ring $\End(E)$. If two elliptic curves $E_1$,$E_2$ are isogenous, written $E_1 \sim E_2$, then $\End^{0}(E_1)\simeq \End^{0}(E_2)$.
An elliptic curve $E$ defined over $\overline{\F}_{p}$ is said to be \emph{supersingular} if the endomorphism algebra, $\End^0_{\overline{\F}_p}(E)$, is isomorphic to a definite quaternion algebra  $\Bp$ over $\Q$ ramified at $p$ and $\infty$. Here, $\Bp = \Q + \Q\mi + \Q\mj + \Q\mi\mj$ for $\mi^2=a$, $\mj^2=b$, and $\mi\mj = -\mj\mi$. In particular, for $p\equiv 11\pmod {12}$ we can consider $(a,b)=(-1,-p)$~\cite[Proposition 1]{EHLMP}.

An abelian variety over $\overline{\F}_{p}$  is said to be \emph{supersingular} if it is isogenous to a product of supersingular elliptic curves over $\overline{\F}_p$ \cite{Oort1974}, and it is said to be \emph{superspecial} if it is isomorphic to a product of supersingular elliptic curves over $\overline{\F}_{p}$  (as an unpolarized abelian variety). A curve $\C$ is called \emph{supersingular} (respectively, \emph{superspecial}) if its Jacobian $\A = \mathcal{J(C)}$ is supersingular (respectively, \emph{superspecial}). 

\begin{theorem}\label{unpolsuperspecial}(Deligne, Ogus, Shioda) If $\A/\overline{\F}_p$ is a superspecial abelian
variety with $\dim \A = g > 1$, then $\A\simeq  E^g$ for any supersingular elliptic curve $E$.
\end{theorem} 
For a proof, see \cite[Theorem 6.2]{Ogus},\cite[Theorem 3.5]{Shioda}, and \cite[Section 1.6]{OortLiModularSpacesofSuperSingular}. 

For $g=1$, there are many superspecial abelian varieties, namely supersingular elliptic curves, each with a principal polarization. In contrast, for $g>1$, a single superspecial abelian variety may admit many principal polarizations.
The map in~\eqref{mainisomorphism} (or using Corollary 2.9 of \cite{ibukiyama1986supersingular}) induces a bijection between principal polarizations  $\theta\in \mathcal{P}(E_1\times E_2)$ and positive definite quaternion unimodular hermitian matrices
\begin{align}\label{eq:sspolar}
    \mathcal{P}(E_1\times E_2) =  \left\{\begin{pmatrix}
        u & \alpha \\ \overline{\alpha} & v
    \end{pmatrix} :u,v\in \Z_{>0}, \ \alpha\in \curlyO,\  uv-\alpha\overline{\alpha}=1\right\}
\end{align}
where $\curlyO$ is a maximal order of $\Bp$.
Throughout, we write $\mathbf{h}=\mathbf{h}(p)$ for the number of supersingular elliptic curves over
$\overline{\F}_p$ up to isomorphism and $\mathbf{H}=\mathbf{H}(p)$ for the number of principally
polarized superspecial abelian surfaces $(\A,\theta)$ up to isomorphism; exactly
$\frac{\mathbf{h}(\mathbf{h}+1)}{2}$ of the latter carry a reducible polarization, so that
$\mathbf{H}-\frac{\mathbf{h}(\mathbf{h}+1)}{2}$ are Jacobians with irreducible polarization \footnote{Closed forms for $\mathbf{h}$ and $\mathbf{H}$ are recalled in \eqref{eq:hcount} and \eqref{eq:Hcount}
of Appendix~\ref{app:iko}.}. For an
integral quadratic form $q$, we write $\Gen{q}$ for its genus, that is, the set of classes locally
equivalent to $q$ at every place. We refer to \cite{edagaurish} for these quantities and for their
computation.
Assume that $\charec(K)=p\geq 3$. By Igusa \cite{Igusa}, every genus $2$ curve $\C$ admits a unique double cover $\C\to\mathbf{P}^1$, whose Galois group is generated by the hyperelliptic involution $\iota$. The quotient
$\mathrm{RA}(\C):=\Aut(\C)/\langle\iota\rangle$ is called the \emph{reduced automorphism group} of $\C$. 

Following \cite{ibukiyama1986supersingular}, for each reduced automorphism group, we now state how many supersingular genus $2$ curves are Jacobians of the product of two supersingular elliptic curves (up to isomorphism). 

\begin{theorem}\emph{\cite[Theorem 3.3]{ibukiyama1986supersingular}}
\label{Ibuyikama-auto-count}
For a prime $p$, let $a_G(p)$ denote the number of isomorphism classes of genus $2$ curves $\C$ such that $\mathcal{J(C)}$ is isomorphic to a product of two supersingular elliptic curves and $\mathrm{RA}(\C)\cong G$. If $p\geq 7$, then $G\in\{1,\,C_2,\,C_2\times C_2,\,S_3,\,D_6,\,S_4,\,C_5\}$, and the explicit values of $a_G(p)$ are reproduced in Appendix~\ref{app:iko}. For $p=5$, there is a unique such curve, with reduced automorphism group $\mathrm{PGL}(2,5)$, while $a_G(p)=0$ for all $G$ when $p=2,3$.
\end{theorem}

\subsection{Isogeny graphs of abelian varieties}\label{sec:isogenygraphs}
From an algorithmic viewpoint, an isogeny class of abelian varieties over a field $K$ is naturally encoded by an \emph{isogeny graph}. Its vertices are the $K$-isomorphism classes $[\A]$, and its edges are $K$-isomorphism classes of isogenies $\A\to\A'$ satisfying certain conditions, mostly on their degree. For a prime $\ell \nmid\charec(K)$, one considers separable isogenies of degree $\ell$; the resulting \emph{$\ell$-isogeny graph} reflects the endomorphism rings. For elliptic curves over finite fields, it is identified with imaginary quadratic orders in the ordinary case and with maximal orders in the quaternion algebra
$\Bp=\left(\frac{-1,-p}{\Q}\right)$, ramified at $p$ and $\infty$ in the supersingular case. Accordingly, ordinary $\ell$-isogeny graphs have a \emph{volcano} structure, whereas supersingular $\ell$-isogeny graphs are \emph{connected regular graphs} described by ideals in maximal quaternion orders. 

From a cryptographic perspective, isogeny graphs have become a fundamental setting for post-quantum constructions, beginning with dimension-one protocols based on random walks, group actions, and path-finding in isogeny graphs of elliptic curves \cite{CGLHash,CRS1,CRS2,CSIDH,SIDH,SIKE}. These developments naturally motivate extending isogeny-graph techniques to higher-dimensional principally polarized abelian varieties. Early work in dimension two was initiated by Charles, Goren, and Lauter~\cite{CharlesGorenLauter2009FamiliesRamanujan}, Takashima \cite{Takashima2018}, and Flynn and Ti \cite{genus2SIDH}. Generalizations of the Charles--Goren--Lauter construction to genus $2$ were proposed by Castryck, Decru, and Smith \cite{Genus2hashfunction} and studied further by Takashima \cite{Takashima2018}. The associated isogeny graphs were subsequently analyzed in detail. Building on the classification of automorphism groups of supersingular genus $2$ Jacobians by Ibukiyama, Katsura, and Oort \cite{ibukiyama1986supersingular}, the local structure of principally polarized superspecial isogeny graphs was determined in \cite{enricsmithisogenyatlas,KT20}. The random-walk behavior and mixing rates were studied experimentally in \cite{floritsmith}. Recently, the paper \cite{isogenygraph2024} introduced finite regular directed graphs of $\ell$-marked principally polarized superspecial abelian varieties of dimension $g$ and proved that their adjacency matrices have real spectra with spectral gaps independent of $p$.

In summary, isogeny graphs of elliptic curves and higher-dimensional abelian varieties over finite fields play an important role in post-quantum cryptography. Their cryptographic relevance relies on the presumed hardness of finding paths between arbitrary vertices. 

The efficiency and security of dimension-$2$ isogeny-based constructions remain largely conjectural, reflecting the limited understanding of supersingular isogeny graphs in higher dimensions. Fundamental properties such as connectedness, diameter, and expansion are not yet fully understood. In particular, rapid mixing of random walks, which underlies the security of several hash constructions and is well established for elliptic curves, remains unclear even for $g=2$.

For supersingular elliptic curves in characteristic $p$, the $\ell$-isogeny graph is connected \cite[Corollary 78]{Kohel}; equivalently, any two vertices are joined by a path of $\ell$-isogenies. One proof uses the representation of sufficiently large integers by primitive integral quaternary quadratic forms, while another relies on the one-dimensionality of the weight-$2$ Eisenstein space for $\Gamma_0(p)$ \cite[p.~223]{Mestre1986Graphes}. In higher dimensions, the connectivity of the $(\ell,\ldots,\ell)$-isogeny graph for supersingular abelian varieties was established via strong approximation for quaternion unitary groups \cite{Jordan,Zaytman,jordanzaytman}. For RM superspecial abelian varieties, connectivity is proved in \cite[Theorem 4]{CharlesGorenLauter2009FamiliesRamanujan}; these graphs are moreover Ramanujan and have few loops and multiple edges \cite[Section 3.2]{CharlesGorenLauter2009FamiliesRamanujan}, properties relevant to cryptographic applications. See also \cite{SurveyLoveGoren2025}.

\textbf{Richelot and $(\ell,\ell)$-isogeny graphs}
Let $\A$ be a principally polarized abelian surface over $K$, and let $\ell\neq\charec(K)$ be a prime. An $(\ell,\ell)$-isogeny is an isogeny $\A\to\A'$ whose kernel is a maximal isotropic subgroup of $\A[\ell]$ with respect to the $\ell$-Weil pairing.
For primes $\ell$ and $p$, the $(\ell,\ell)$-isogeny graph $\Gamma_2(\ell;p)$ is the weighted multi-graph whose vertices are isomorphism classes of principally polarized abelian surfaces in a fixed isogeny class over $\overline{\mathbb{F}}_p$, and whose edges correspond to isomorphism classes of $(\ell,\ell)$-isogenies. Since an $(\ell,\ell)$-isogeny admits a dual, the graph is undirected once edges are weighted by the number of isomorphism classes of kernels realizing them, and it is regular of weighted degree $(\ell+1)(\ell^2+1)$. 

The $(\ell,\ell)$-isogeny graph of superspecial principally polarized abelian surfaces over $\overline{\mathbb{F}}_p$ is denoted by $\Gamma_2^{SS}(\ell;p)$. When the parameters $\ell$ and $p$ are clear from the context, we call $\Gamma_2^{SS}(\ell;p)$ the \emph{superspecial graph}. When $\ell=2$, we say that $\Gamma_2^{SS}(2;p)$ is a \emph{Richelot} superspecial isogeny graph. 

Another important aspect of the structure of isogeny graphs is the behavior of \emph{random walks}. We therefore recall several results concerning random walks on isogeny graphs. The supersingular $\ell$-isogeny graph is a constant-degree Ramanujan expander graph \cite{pizer1998}; consequently, a random walk on such a graph mixes rapidly. For any starting vertex, after $ O(\log(p))$ steps, the distribution of the endpoint is $\varepsilon$-close to the uniform distribution on vertices, where the hidden constant depends only on the spectral gap (equivalently, on $\ell$) and on $\varepsilon$.
To discuss random walks on superspecial isogeny graphs in dimension $2$, we recall the relevant spectral bounds. Let $G$ be an undirected graph that is regular of weighted degree $d$, in the sense just described, and denote by $\lambda_\star(G)$ its normalized second-largest adjacency eigenvalue. Florit-Smith conjecture that, for every prime $\ell$, there exists a constant $\lambda=\lambda(2,\ell)<1$, independent of $p$, such that
$\lambda_\star\!\left(\Gamma^{SS}_2(\ell;p)\right)\leq \lambda$
for every prime $p\geq 5$. In the Richelot case $\ell=2$, they conjecture that
$
\frac{11}{15}\leq \lambda_\star\!\left(\Gamma^{SS}_2(2;p)\right)\leq \frac{12}{15}
$
for every $p\geq 41$ \cite[Conjecture~4.10]{floritsmith}. This uniform spectral gap conjecture implies rapid mixing of random walks on the corresponding superspecial Richelot isogeny graphs. 

More precisely, assuming the conjecture, Florit--Smith show that a random walk of length $n\geq 4.5m\log p+9$ approximates the stationary distribution with error at most $p^{-m}$; taking $m=4$, a walk of length $n\geq 18\log p+9$ therefore gives error at most $p^{-4}$ \cite[Theorem~6.1]{floritsmith}. We use this only as a qualitative indication that logarithmic-length random walks are expected to mix rapidly.

Another useful piece of information at hand is the \emph{diameter} of the $(\ell,\ell)$ isogeny graph of principally polarized superspecial surfaces for a specific $\ell$ and $p$. 
Florit-Smith also gave a result on the diameter of such isogeny graphs in the generic case \cite{floritsmith} and calculated them explicitly for small primes; see Table~\ref{tab:richelot_diameter_p_11mod12}. For a given fixed $p$ and $\ell$, they concluded
that the diameter of each isogeny graph $\Gamma^{SS}_g(\ell;p)$ is $O(\log(p))$, see \cite[p.12]{floritsmith}.
While searching for a number represented by a refined Humbert invariant in our algorithm, we can use the upper bound on the diameter to put an upper bound on the number represented by the quadratic form.
\begin{table}[!ht]
\centering
\footnotesize
\renewcommand{\arraystretch}{0.95}
\setlength{\tabcolsep}{3pt}
\begin{tabular}{r|rrrrrrrrrrrrrr}
\hline
$p$
& 23 & 47 & 59 & 71 & 83 & 107 & 131
& 167 & 179 & 191 & 227 & 239 & 251 & 263 \\
$d(G)$
& 3 & 4 & 5 & 5 & 6 & 6 & 7
& 7 & 7 & 7 & 7 & 8 & 8 & 7 \\
\hline
$p$
& 311 & 347 & 359 & 383 & 419 & 431 & 443
& 467 & 479 & 491 & 503 & 563 & 587 & 599 \\
$d(G)$
& 8 & 8 & 8 & 8 & 9 & 8 & 8
& 8 & 8 & 8 & 9 & 9 & 9 & 9 \\
\hline
\end{tabular}
\caption{Experimental diameters $d(G)$ of the superspecial Richelot isogeny graph
$G=\Gamma^{SS}_2(2;p)$ for primes $p\equiv 11 \pmod{12}$ (data extracted from \cite[Appendix A]{floritsmith}).}
\label{tab:richelot_diameter_p_11mod12}
\end{table}

A useful refinement is obtained by combining spectral bounds with general diameter estimates. Let $M=\#\Gamma^{SS}_g(\ell;p)$, and suppose that the absolute values of the non-trivial normalized adjacency eigenvalues are bounded above by a constant strictly smaller than $1$. The usual logarithmic estimate then gives an upper bound for the diameter in terms of $M$ and the spectral bound, while the Chebyshev-polynomial argument of Chung--Faber--Manteuffel \cite{ChungFaberManteuffel1994} gives the sharper estimate
$\operatorname{diam}(\Gamma^{SS}_g(\ell;p))\leq \left\lfloor \frac{\operatorname{arcosh}(M-1)}{\operatorname{arcosh}\left(\frac{1}{\lambda_{\max}}\right)}\right\rfloor+1$,
where $\lambda_{\max}$ denotes the corresponding upper bound for the non-trivial normalized eigenvalues. For $g=2$ we have $M=\mathbf{H}(p)$, given in closed form by \eqref{eq:Hcount}; for instance $\mathbf{H}(1103)=480508$. In general $M$ is of size $p^{\frac{g(g+1)}{2}}$, so this recovers a logarithmic diameter bound in $p$. In dimension $2$, the estimate of Oh \cite{Oh2002} gives $\lambda_{\max}\leq\left(\frac{2\sqrt{\ell}}{(\ell+1)}\right)^2$. For the Richelot isogeny graph $\Gamma^{SS}_2(2;p)$, this gives $\lambda_{\max}\leq\frac{8}{9}$, and the  estimate of Chung--Faber--Manteuffel gives
$\operatorname{diam}(\Gamma^{SS}_2(2;p))\leq4.20\log_2(p)+O(1)$.
This bound is unconditional. Assuming instead the upper bound $\lambda_\star(\Gamma^{SS}_2(2;p))\leq\frac{12}{15}=\frac45$ of \cite[Conjecture~4.10]{floritsmith} quoted above, the same estimate gives the sharper but \emph{conditional} bound
$\operatorname{diam}(\Gamma^{SS}_2(2;p))\leq3\log_2(p)+O(1)$.

\subsection{Refined Humbert invariants}
Here, we shortly give the definition of a refined Humbert invariant of principally polarized abelian surface. For the details, we refer the reader to Appendix~\ref{sec:RHI}.

Let $\A$ be an abelian surface over a field $K$.
Let $\Div(\A)$ be the set of divisors of $\A$.
If $D_1$, $D_2 \in \Div(\A)$ we say that $D_1$ is \emph{numerically equivalent} to $D_2$, denoted by $D_1 \equiv D_2$, satisfying
$(D_1\cdot D) = (D_2\cdot D)$ for all $D \in \Div(\A)$,
where $(\cdot)$ denotes the intersection number (as defined in \cite[p.357]{hartshorne2013algebraic}). We define the \emph{N\'eron-Severi group} $\NS(\A)$ of $\A$ to be $\NS(\A) = \Div(\A)/\equiv$, as defined in \cite[Corollary of Theorem V.1]{langabelianvar}.
We define \emph{the set of principal polarizations} of  $\A$ as $$ \mathcal{P}(\A)= \{D\in \NS(\A) : D\in \Div(\A) \emph{ is ample and } D\cdot D =2 \}.$$

Let $\A/K$ be an abelian surface and let $\theta \in \mathcal{P}(\A)$ be a principal polarization. We define the \emph{polarized N\'eron-Severi group} of $(\A,\theta)$ to be  $\NS(\A,\theta) :=\NS(\A)/\Z\theta$.

Kani discusses in \cite[\S 3]{Kani94} that there is a well-defined map on $\NS(\A)$, and this leads to a quadratic form $\tilde{q}_{(\A,\theta)}$ on $\NS(\A)$, given explicitly by the formula
$$\tilde{q}_{(\A,\theta)}(D) = (D \cdot \theta)^2 - 2(D \cdot D) for  D \in \NS(\A).$$
It is evident that $\tilde{q}_{(\A,\theta)}(D + n\theta) = \tilde{q}_{(\A,\theta)}(D)$ for all $n\in \Z$. Consequently, $\tilde{q}_{(\A,\theta)}$ induces a defines a positive-definite quadratic form $q_{(\A,\theta)}$ on the quotient module $\NS(\A,\theta)$.\footnote{This can be seen as an integral quadratic form on $\Z^{\rho-1}$ since $\NS(\A,\theta) \simeq Z^{\rho-1}$, where $\rho$ is the Picard number of $\A$.}
\begin{definition}\label{RHIdefn}
    Let $(\A,\theta)$ be a principally polarized abelian surface. A
    \emph{refined Humbert invariant}\footnote{Note that a refined Humbert invariant is considered
    up to $\GL_n(\Z)$ equivalences.} $q_{(\A,\theta)}$ of $(\A,\theta)$ is the positive-definite
    integral quadratic form in $\rho-1$ variables on $\NS(\A,\theta)$ given by
    \[q_{(\A,\theta)}(D+\Z\theta) = (D \cdot\theta)^2 - 2(D\cdot D),\qquad D\in\NS(\A).\]
\end{definition}

The main use of a refined Humbert invariant in this paper is Kani's splitting theorem, which
converts the splitting of an abelian surface into a primitive representation.
\begin{theorem}\emph{\cite[Theorem 20]{KaniESCI}}\label{SplittingTheorem}
A curve $\C/K$ has an elliptic\footnote{Theorem~\ref{SplittingTheorem} incorporated into our
algorithms works on genus $2$ curves lying on the product surface, that is to mean
$\mathcal{J(C)}\cong E_1 \times E_2$ for any two elliptic curves $E_1$ and $E_2$ (not necessarily
supersingular elliptic curves).} subcover of degree $N$ if and only if $q_{\C}$ primitively
represents $N^2$. Equivalently, $\mathcal{J(C)}$ is $(N,N)$ split if and only if $q_{\C}$
primitively represents $N^2$.
\end{theorem}

\subsection{Primitive representations by positive definite quadratic forms}
\label{chap:integral-local-global}
A \emph{quadratic lattice} is a pair $(\Z^n,q)$, where $q$ is a quadratic form. A \emph{representation} of $(\Z^m,q)$ by $(\Z^n,q_0)$ is an isometry $\eta:(\Z^m,q)\to(\Z^n,q_0)$. It is \emph{primitive} if $\eta(\Z^m)$ is a primitive sublattice of $\Z^n$, equivalently, $\eta(\Z^m)=(\eta(\Z^m)\otimes_{\Z}\Q)\cap\Z^n$. We say that $(\Z^m,q)$ is \emph{locally primitively representable} by $(\Z^n,q_0)$ if such a primitive representation exists over $\Z_\nu$ for every prime $\nu$ and over $\mathbb{R}$.

The integral local--global problem asks when local primitive representability implies primitive representability over $\Z$. For indefinite target forms and $n\ge m+3$, this follows from Eichler's strong approximation for spin groups~\cite{Eichler1952Aehnlichkeitsklassen}. The positive definite case is subtler. Hsia, Kitaoka, and Kneser~\cite{HsiaKitaokaKneser1978} proved a local--global theorem for $n\ge 3m+3$ when the minimum of the represented lattice is sufficiently large, and J\"ochner and Kitaoka~\cite{JochnerKitaoka1994} reduced the dimension requirement to $n\ge 2m+3$ under additional hypotheses. Ellenberg and Venkatesh~\cite{EllenbergVenkatesh2008} subsequently obtained an ineffective result for $n\ge m+5$, and also for $n\ge m+3$ when the minimum is sufficiently large. A recent effective result gives the following form of the principle in codimension at least~$3$.

\begin{theorem}[{\cite[Theorem 1.3]{EinsiedlerLindenstraussMohammadiWieser2025}}]
\label{thm:integralHasse}
Let $m,n\in\Z_{>0}$ with $n\ge m+3$. There exist constants $C,A>0$, depending only on $n$, such that the following holds. Let $(\Z^m,q)$ and $(\Z^n,q_0)$ be positive definite quadratic lattices. If $(\Z^m,q)$ is locally primitively representable by $(\Z^n,q_0)$ and $\min_{\{0\neq x\}}q(x)\ge C\disc(q_0)^A$, then $(\Z^m,q)$ is primitively representable by $(\Z^n,q_0)$.
\end{theorem}

For $m=1$, this gives a local-global principle for primitive representations of sufficiently large integers by positive definite forms; classical results in this direction go back to Kloosterman and Tartakowsky (see~\cite[\S 11]{Iwaniec1997}). Related work of Kitaoka~\cite{Kitaoka1989SomeRemarks,Kitaoka1994MinimumPrimitiveI,Kitaoka1996MinimumPrimitiveII} studies more generally the relation between primitive and non-primitive representation problems.

\subsection{Related work on splittings}
It is important to distinguish three computational formulations of the $(N,N)$-splitting problem for a principally polarized superspecial abelian surface $(\A,\theta)$: (i) determining a splitting degree $N$ for which $(\A,\theta)$ admits an $(N,N)$-splitting; (ii) for a prescribed $N$, deciding or certifying whether such an $(N,N)$-splitting exists; and (iii) constructing the corresponding explicit geometric isogeny $\A\to E_1\times E_2$. Existing approaches address different parts of this problem. 

Corte-Real Santos--Costello--Frengley \cite{Splitting,detectionofendomorphisms} work geometrically and, for a prescribed integer $N\leq 11$, efficiently test whether a superspecial genus $2$ Jacobian is optimally $(N,N)$-split using the Humbert surface $\mathcal{H}_{N^2}$. Their test certifies the existence of such a splitting but does not construct the corresponding $(N,N)$-isogeny. They incorporate these fixed-$N$ tests into \emph{SplitSearcher}, which performs a walk in the Richelot $(2,2)$-isogeny graph and uses splitting detection to accelerate the search for a suitable endpoint. Thus, this approach addresses the fixed-$N$ detection problem rather than (i), in which the candidate values of $N$ are chosen in advance rather than determined intrinsically from $(\A,\theta)$. 

Kunzweiler--Shen \cite{SabrinaSplitting} instead study the computational consequences of having a splitting. They reduce the problem of finding a non-scalar endomorphism of a supersingular elliptic curve to a restricted \emph{good splitting problem}, in which the returned splitting is required to satisfy an additional compatibility condition with a preceding isogeny, and give a proof-of-concept implementation based on a search in the isogeny graph via theta functions and Kani's lemma. Kutas--Shen \cite{petersplitting} subsequently eliminates the need for this restriction at the level of the security reduction: they give an oracle reduction from the non-scalar endomorphism problem to the unrestricted splitting problem, in which the oracle may return an arbitrary splitting. They further derive a heuristic equivalence between a suitably restricted degree splitting problem and the problem of computing the full endomorphism ring of a principally polarized superspecial abelian surface. These results clarify the computational relation between splitting and endomorphism computation, but treat the splitting problem itself as an oracle: they do not provide an algorithm for determining a splitting degree or constructing the corresponding explicit splitting isogeny.

On the algebraic representation,  $\mathrm{KLPT}^2$~\cite{KLPT^2} is particularly relevant since its input consists of IKO matrices (representing principal polarizations) on superspecial abelian surfaces. Thus, the IKO representation retains precisely the polarization data, and needed for polarized dimension $2$ isogeny computations. This is the same type of polarized algebraic input from which the refined Humbert invariant is obtained in \cite{edagaurish}. Page--Robert--Soumier \cite{AurelJulienDamien} are especially relevant in this aspect, as they use splittings as an ingredient to improve $\mathrm{KLPT}^2$. Their  interpretation Hermitian modules allows them to work explicitly with orthogonal decompositions into rank-$1$ projective submodules and to construct smooth splitting similitudes from rank-$2$ modules to split modules. The corresponding polarized isogenies $\A_i\to E_0\times E_i$ are then combined with $\mathrm{KLPT}$ algorithm \cite{KLPT} of elliptic curves, reducing the output degree of $\mathrm{KLPT}^2$ from $O(p^{25})$ to $O(p^{15.5})$. Their result, therefore, highlights the importance of splitting isogenies for improving the efficiency of algebraic dimension-2 isogeny path-finding problem.

The refined Humbert invariant provides a direct arithmetic formulation of this splitting detection problem. From the algebraic representation of $(\A,\theta)$, one obtains the quintic refined Humbert invariant $q_{(\A,\theta)}$ , and  the $(N,N)$-splitting is detected by a primitive representation $q_{(\A,\theta)}(x)=N^2$. Thus, the splitting degree $N$ itself can be searched for through primitive square representations of a positive definite integral quadratic form in five variables obtained from a principal polarization encoded by a IKO matrix.

\section{Algorithms} \label{designing}
\subsection{Constructing random polarizations} 
Let $p \equiv 11 \pmod{12}$ be prime, let $\Bp = (-1,-p \mid \Q)$ be the definite quaternion
algebra over $\Q$, and fix the maximal order
$\Ord = \Z\langle 1, \beta_1, \beta_2, \beta_3 \rangle
       = \Z\left\langle 1,\ i,\ \tfrac{i+j}{2},\ \tfrac{1+ij}{2} \right\rangle
       \;\cong\; \End(E)$,  $j(E) = 1728$.
A principal polarization on the superspecial surface $\A = E \times E$ is then a positive definite
unimodular Hermitian matrix \cite[Equation~(11)]{edagaurish}
\[
  \theta = \begin{pmatrix} u & \alpha \\ \bar\alpha & v \end{pmatrix}, \quad u,v \in \Z_{>0},\ \alpha \in \Ord,
  \quad \det\theta = uv - \nrd(\alpha) = 1 ,
\]
and the object of this subsection is to sample such a $\theta$ at random, for $p$ far larger than
exhaustive enumeration can reach.

The sampler does not work directly on an order basis. It works instead in the half-coordinates
$\alpha = \frac{1}{2}(A_0 + A_1 i + A_2 j + A_3 ij)$, for which the reduced norm is
$\nrd(\alpha) = \frac{1}{4}\left(A_0^2 + A_1^2 + p(A_2^2+A_3^2)\right)$ and membership
$\alpha \in \Ord$ amounts to the two congruences $A_0 \equiv A_3$ and $A_1 \equiv A_2 \pmod 2$.
Setting $z = A_3$, $y = A_2$, $w = \frac{A_0-z}{2}$ and $x = \frac{A_1-y}{2}$ converts a solution
back to $\alpha = w + x\beta_1 + y\beta_2 + z\beta_3$, which is the $[u,v,w,x,y,z]$ record
format of the generated data.

As a construction, this is not new, and it is worth being precise about what is borrowed. The
overall shape --- draw $u$ and $v$, set $n = uv-1$, and hand $n$ to a norm-equation solver --- is
$\mathrm{KLPT}^2$'s \texttt{RandomPolarisation}~\cite{KLPT^2}. The solver it calls is the
Kohel--Lauter--Petit--Tignol \texttt{RepresentInteger}~\cite{KLPT}, in the form used by SQIsign~\cite{SQIsign}. That the resulting matrices are exactly the principal polarizations is due to
Ibukiyama--Katsura--Oort \cite[Corollary~2.9]{ibukiyama1986supersingular}, restated as \cite[Theorem~2.8]{KLPT^2}. What is new here is a restriction on the search, which the following observation makes available.

\begin{lemma}\label{lem:parity}
  Let $p \equiv 3 \pmod 4$, let $n \in \Z$, and set $\rhs = 4n - p(A_2^2+A_3^2)$. Then
  \[
    \rhs \equiv
    \begin{cases}
      0 \pmod 4 & \text{$A_2,A_3$ both even}, \\
      1 \pmod 4 & \text{$A_2,A_3$ of opposite parity}, \\
      2 \pmod 4 & \text{$A_2,A_3$ both odd}.
    \end{cases}
  \]
  In particular the condition $\rhs \equiv 1 \pmod 4$ required by \texttt{Cornacchia}~\cite{cornacchia} holds exactly for
  opposite-parity pairs $(A_2,A_3)$, and this is known before $\rhs$ is computed.
\end{lemma}

\begin{proof}
  A square is $0$ or $1$ modulo $4$ according as its root is even or odd, so $A_2^2+A_3^2$ is
  congruent to $0$, $1$ or $2$ in the three respective cases. Since $p \equiv 3 \pmod 4$ we have
  $-p \equiv 1$ and $-2p \equiv 2 \pmod 4$, while $4n \equiv 0$, so that
  $\rhs \equiv -p(A_2^2+A_3^2) \pmod 4$ takes the stated values.
\end{proof}

Restricting the search to opposite-parity pairs, therefore, throws away half of every draw that
could never have reached the Cornacchia step (\lin{lin:corn}), and none that could. With that in hand, the whole sampler fits in one pass, as Algorithm~\ref{alg:polzrandom}.

\begin{algorithm}[H]
  \caption{\texttt{RandomPolarizations}: sample $T$ principal polarizations}
  \label{alg:polzrandom}
  \Input{prime $p \equiv 11 \pmod{12}$; exponent $\mathtt{sbound}$; sample size $T$}
  \Output{a list $\Theta$ of $T$ matrices $\theta = \theHmat$ with $u,v \in \Z_{>0}$, $\alpha \in \Ord$, $\det\theta = 1$}
  $\Theta \gets \varnothing$\;
  \While{$|\Theta| < T$}{
    $u \gets$ uniform in $[1,\, p^{\mathtt{sbound}}]$\label{lin:u}\;
    $v \gets$ uniform in $[1,\, p^{3\,\mathtt{sbound}}]$\;
    $n \gets uv - 1$\label{lin:N}\;
    $B \gets \lfloor \sqrt{4n/p} \rfloor$\label{lin:B}\;
    \Repeat{$\rhs > 0$ \KwAnd $\rhs$ is prime\label{lin:until}}{
      sample $(A_2,A_3)$ uniformly in $[-B,B]^2$ of opposite parity\label{lin:sample}\;
      $\rhs \gets 4n - p\,(A_2^2 + A_3^2)$\label{lin:rhs}\;
    }
    \tcp*[l]{a draw exhausting its retry budget is abandoned; redraw $u,v$}
    $(A_0, A_1) \gets$ \Cornacchia{$\rhs$}\label{lin:corn}\;
    order $A_0, A_1$ so that $A_0 \equiv A_3$ and $A_1 \equiv A_2 \pmod 2$\label{lin:order}\;
    $\alpha \gets \tfrac{1}{2}(A_0 + A_1 i + A_2 j + A_3 ij)$\label{lin:alpha}\;
    append $\theta = \theHmat$ to $\Theta$\label{lin:ret}\;
  }
  \Return $\Theta$\;
\end{algorithm}

The asymmetry of the bounds in \lins{lin:u}{lin:N} is deliberate: letting $v$ range over
$p^{3\,\mathtt{sbound}}$ while $u$ ranges only over $p^{\mathtt{sbound}}$ makes $n = uv-1$ vary
widely relative to $p$, which keeps the supply of $\rhs$ passing the test in \lin{lin:until}
plentiful. \Lin{lin:B} then bounds $(A_2,A_3)$ so that $\rhs$ cannot go negative. The loop of
\lins{lin:sample}{lin:until} is the search proper, and Lemma~\ref{lem:parity} is what makes it
cheap: because \lin{lin:sample} samples only opposite-parity pairs, the congruence
$\rhs \equiv 1 \pmod 4$ holds automatically, so primality is the sole condition left to test, and
\lin{lin:corn} then recovers $\rhs = A_0^2 + A_1^2$ by Cornacchia's algorithm~\cite{cornacchia} without factoring
anything. The pair is redrawn afresh on each pass rather than swept in a fixed order, so a given
$n$ need not always produce the same $\alpha$, and a corpus generated this way does not
concentrate on small $(A_2,A_3)$. \Lin{lin:order} uses the one remaining freedom, the choice of
which of $A_0,A_1$ plays which role, to meet the membership congruences. This always succeeds: $\rhs$
is an odd prime, so exactly one of $A_0,A_1$ is even, while $A_2,A_3$ have opposite parity by
Lemma~\ref{lem:parity}; hence exactly one of the two orderings satisfies $A_0\equiv A_3$ and
$A_1\equiv A_2\pmod2$, and the line has no failure branch.
\Lins{lin:alpha}{lin:ret} assemble $\theta$, and since $\nrd(\alpha) = n = uv-1$ by construction, its determinant is $1$ without any verification step.

A word on what is being sampled. Algorithm~\ref{alg:polzrandom} draws Hermitian \emph{matrices},
not isomorphism classes, and the fibres of the map from matrices to classes are not of equal size:
distinct $\theta$ can give isomorphic $(\A,\theta)$, and the number that do is governed by the
automorphisms of the surface. The distribution induced on isomorphism classes is therefore not
uniform, and we do not determine it here.

What the algorithm does produce, with overwhelming probability, is an \emph{irreducible}
polarization, which is the case of interest for splitting. Of the $\mathbf{H}$ principal
polarizations on $E\times E$, exactly $\frac{\mathbf{h}(\mathbf{h}+1)}{2}$ are reducible
\cite[Table 1]{edagaurish}, so by \eqref{eq:redfrac} the reducible proportion is $10/p+O(p^{-2})$
and tends to zero. A reducible draw is in any case detected and discarded by
Proposition~\ref{irreducibilitycriteria}, so this proportion measures only how rarely a draw is
wasted. What does bear on the experiments is whether the distribution induced on the isometry
classes of \emph{irreducible} polarizations distorts the statistics we read off them; that is
tested in Section~\ref{data} against the lists of \cite{edagaurish}.

Three practical matters are suppressed above. The search in \lins{lin:sample}{lin:until} carries a
retry budget, and a draw that exhausts it is abandoned and restarted with fresh $u,v$; the budget
grows as $\Theta(\mathtt{sbound} \cdot \log p)$, which matches the $O(\log p)$ samples one expects
to need before a prime $\rhs$ appears. For small $p$ the implementation will also accept a
composite $\rhs$ that it can factor outright, which brings the same-parity cases of
Lemma~\ref{lem:parity} back into play; that path is unreachable at cryptographic size, and the
restriction in \lin{lin:sample} is imposed only once it is provably dead. Finally,
$\mathtt{sbound}$ is not a constant here. $\mathrm{KLPT}^2$ fixes $\mathtt{sbound} = 20$, so that
$n \approx p^{80}$, which is harmless for a small test prime but runs to some $80{,}000$ bits for
a $1000$-bit one; choosing $\mathtt{sbound} \approx \mathtt{target\_bits}/(4\log_2 p)$ instead
holds the norm equation near a fixed size across the whole range.

It is worth setting Algorithm~\ref{alg:polzrandom} beside \cite[Algorithm~2]{edagaurish}, which
answers a different question. That algorithm, \texttt{polz}$(p, V_{\max}, U_{\max}, T_{\max})$,
enumerates principal polarizations exhaustively and classifies them up to isomorphism, growing its
bounds until the count reaches $\mathbf{H}(p)$. That stopping condition is met, and the list is
then complete, for $p\leq191$; at $p=227,239,251$ the published runs halt short of it
\cite[Table 1]{edagaurish}, and in each case every polarization missed is reducible, so the
irreducible ones --- the case of interest here --- are complete throughout $p\leq251$.
Algorithm~\ref{alg:polzrandom} surrenders both the completeness and the
classification, returning raw polarizations rather than representatives, and buys with that the
ability to keep going at sizes where enumeration has long since stopped. The two are
complementary rather than interchangeable.

\subsection{Primitive representations of $4^k$ and $\ell^{2k}$}\label{sec:primreps}

We now specialize to the case needed for quintic refined Humbert invariants below, following \cite[Section 3.1]{edagaurish}. Fix an odd prime $p\equiv 11\pmod{12}$ and let
\[
q_5(x_0,x_1,x_2,x_3,x_4)
\;=\; x_0^2+4x_1^2+4x_1x_3+4x_2^2+4x_2x_4+(p+1)x_3^2+(p+1)x_4^2 ,
\quad L_5:=(\Z^5,q_5).
\]
For a quadratic form $q$ in $n$ variables we write $G_q$ for the symmetric
matrix with $q(x)=x^{t}G_qx$, so that $\disc(q)=\det(2G_q)=2^n\det G_q$. In
the order of variables $x_0;\,x_1,x_3;\,x_2,x_4$ the matrix $G_{q_5}$ is block
diagonal,
\[
G_{q_5}=(1)\oplus\begin{pmatrix}4&2\\2&p+1\end{pmatrix}\oplus
\begin{pmatrix}4&2\\2&p+1\end{pmatrix},
\quad
\det G_{q_5}=1\cdot4p\cdot4p=16p^2,
\quad
\disc(q_5)=2^5\cdot16p^2=2^9p^2 ,
\]
and $L_5$ is positive definite. All lattices $L=(\Z^5,q)$ of interest
below lie in the genus of $L_5$, i.e.\ $q$ and $q_5$ agree up to a change of
variables in $\GL_5(\Z_\nu)$ for every prime $\nu$; in particular
$\Delta_L=\disc(q)=2^9p^2$.
\paragraph{Represents locally.} For $m=1$ the lattice to be represented is $(\Z,tz^2)$ with
$t\in\Z_{>0}$. A representation $\eta$ of $(\Z,tz^2)$ by
$(\Z^5,q)$ over $\Z_\nu$ is $\eta(z)=z\alpha$ with $\alpha=\eta(1)\in
\Z_\nu^5$ and $q(\alpha)=t$, and it is primitive exactly when
$\alpha\notin\nu\Z_\nu^5$; over $\mathbb{R}$ a representation exists for
every $t>0$. A change of variables in $\GL_5(\Z_\nu)$ maps
$\Z_\nu^5\setminus\nu\Z_\nu^5$ onto itself, so the existence
of such $\alpha$ depends only on the genus of $(\Z^5,q)$. 

\begin{theorem}\label{thm:local-automatic}
Let $p$ be an odd prime, let $L=(\Z^5,q)$ be a quadratic lattice in
the genus of $L_5$, and $t\in\Z_{>0}$.
\begin{enumerate}
\item[(i)] For every odd prime $\nu$, $(\Z,tz^2)$ is primitively
representable by $L$ over $\Z_\nu$.
\item[(ii)] If $p\equiv3\pmod4$, then $(\Z,tz^2)$ is primitively
representable by $L$ over $\Z_2$ if and only if $t\equiv0$ or
$1\pmod4$.
\end{enumerate}
Consequently, if $p\equiv3\pmod4$, then $(\Z,tz^2)$ is locally
primitively representable by $L$ if and only if $t\equiv0$ or $1\pmod4$;
in particular for every square $t$, hence for $t=\ell^{2k}$ with $\ell$
any prime and $k\ge0$, and for $t=4^k$.
\end{theorem}

The proof is in Appendix~\ref{app:proof}.
\paragraph{Represents globally.} The \emph{shell} of norm $m$ is $\{x\in\Z^{5}: q(x)=m\}$;
\[
r(L,m)=\#\{x\in\Z^{5}: q(x)=m\},\quad
r^{*}(L,m)=\#\{x\in\Z^{5}: q(x)=m,\ \gcd(x)=1\}
\]
are its size and its primitive part. The lattice minimum
$\lambda_{1}(L)=\min\{q(x):0\ne x\in\Z^{5}\}$ is a squared length in this normalization. If $\ell^{2k}<\lambda_{1}(L)$ the shell is empty, and $\lambda_{1}(L)\le8^{1/5}(\det G_q)^{1/5}$
by Hermite, so the search starts at
\begin{equation}\label{eq:k1bound}
  k_1=\left\lceil\log_{\ell^{2}}\lambda_{1}\right\rceil\;\lesssim\;\tfrac15\log_{\ell}p .
\end{equation}
For a class $L$, put
$k_{0}(L)\;=\;\min\{\,k\ge1 \;:\; r^{*}(L,\ell^{2k})>0\,\}$,
the least level at which $L$ primitively represents $\ell^{2k}$; equivalently, the least $k$ for which $(\A,\theta)$ admits an $(\ell^{k},\ell^{k})$-splitting. For $k\ge0$,
\begin{equation}\label{eq:mobius}
r(L,\ell^{2k})\;=\;\sum_{i=0}^{k}r^{*}(L,\ell^{2i}),\quad\text{hence}\quad
k_{0}(L)\le k\iff r(L,\ell^{2k})>0 .
\end{equation}
Moreover the shell of norm $\ell^{2k_{0}}$ consists of primitive vectors whenever $\lambda_{1}(L)>1$: if $q(x)=\ell^{2k_{0}}$ and $g=\gcd(x)>1$, then $g\mid\ell^{k_{0}}$, so $x/g$ has norm $\ell^{2i}$ with $i<k_{0}$, which contradicts the minimality of $k_{0}$ when $i\ge1$ and contradicts $\lambda_{1}(L)>1$ when $i=0$. Note that the term $i=0$ in \eqref{eq:mobius} is nonzero precisely when $q$ represents $1$, that is, precisely when $\theta$ is reducible by Proposition~\ref{irreducibilitycriteria}. Representation is therefore monotone in $k$ by \eqref{eq:mobius} while
primitive representation is not: a class may primitively represent
$\ell^{2k_{0}}$ and yet represent $\ell^{2(k_{0}+1)}$ only imprimitively.

\begin{corollary}\label{cor:local-automatic}
Let $p\equiv3\pmod4$ and let $L=(\Z^5,q)$ lie in the genus of $L_5$.
With $C_5,A_5$ the constants of Theorem~\ref{thm:integralHasse} for $m=1$, $n=5$,
every integer $t\equiv0,1\pmod4$ with $t\ge C_5\,\Delta_L^{A_5}$ is
primitively represented by $L$. In particular $\ell^{2k}$ is primitively
represented by $L$ as soon as
\[
k\;\ge\;\tfrac12\log_\ell C_5+\tfrac{A_5}{2}\log_\ell \Delta_L
\;=\;\tfrac12\log_\ell C_5+\tfrac{9A_5}{2}\log_\ell 2+A_5\log_\ell p ,
\]
so the threshold satisfies $k_0(L)=O(\log p)$. The constants $C_5,A_5$ are not made
explicit in Theorem~\ref{thm:integralHasse}, so the implied constant here is unknown and this bound
cannot serve as a computational stopping rule; the rule we actually use is the diameter bound
of Section~\ref{sec:isogenygraphs}, a logically independent statement that happens to agree in order
of magnitude.
\end{corollary}

\begin{proof}
By Theorem~\ref{thm:local-automatic}, $(\Z,tz^2)$ is locally primitively
representable by $L$, and $\min_{0\ne z\in\Z}tz^2=t$; apply
Theorem~\ref{thm:integralHasse} with $m=1$, $n=5$, and put $t=\ell^{2k}$,
$\Delta_L=2^9p^2$.
\end{proof}

Therefore, the whole computation is done by the Kannan-Fincke-Pohst enumeration algorithm~\cite{kannan,FinckePohst,SchnorrEuchner}. 

\begin{algorithm}[ht]
\caption{\texttt{FirstLevel}$(G_q,\ell,K_{\max})$: the least $k$ with a primitive vector of norm $\ell^{2k}$}
\label{alg:kzero}
\Input{The Gram matrix $G_q$ of a lattice $L=(\Z^5,q)$ in the genus of $L_5$; the prime $\ell$; a cap $K_{\max}$.}
\Output{$k_0(L)$ and a primitive $x\in\Z^{5}$ with $q(x)=\ell^{2k_0}$, or ``$k_0>K_{\max}$''.}
$(G',U)\leftarrow$\LLL{$G_q$}\;
$\lambda_{1}\leftarrow$\ShortestVector{$G'$} \tcp*{Kannan-Fincke-Pohst}
$k_{1}\leftarrow$ least $k\ge1$ with $\ell^{2k}\ge\lambda_{1}$\; 
\For{$k=k_{1},\dots,K_{\max}$}{
  \ForEach{$z$ in \Enumerate{$G',\ell^{2k}$}}{ 
    \If{$\gcd(z)=1$}{
      $x\leftarrow Uz$\;
      \Verify{$x^{t}G_qx=\ell^{2k}$ and $\gcd(x)=1$} in exact integer arithmetic; abort on failure\;
      \Return{$(k,x)$}
    }
  }
}
\Return{$k_0>K_{\max}$}
\end{algorithm}

\paragraph{Bounds for the search.}
The diameter bounds of Section~\ref{sec:isogenygraphs} give a stopping criterion for our
representation search. For any fixed $p$ a rigorous bound is obtained by substituting the exact
value $M=\#\Gamma^{SS}_2(2;p)$ into the Chung--Faber--Manteuffel inequality. Once a diameter bound
$\mathsf{Diam}$ is established, it is enough in our search to consider $k\leq \mathsf{Diam}$, and
hence refined Humbert invariant values of the form $2^{2k}=4^k$ with $k\leq \mathsf{Diam}$.

That bound lies far above what the experiments need, and in practice $K_{\max}$ is set from $p$
directly. The small-prime runs of Section~\ref{data} take $K_{\max}=\lceil\ln p\rceil$, which is
$3$ at $p=11$ and $6$ at $p=251$, and decide every level by exhaustive enumeration, so a negative
answer at a given $k$ is a proof rather than a timeout. The $50$-bit run of
Algorithm~\ref{alg:kzero} uses $K_{\max}=\lceil 2\ln p/(3\ln\ell)\rceil=33$, twice the heuristic
value of $k_0$; it is never reached, the largest $k_0$ observed over the $10{,}000$ classes being
$17$.

\subsection{Random primitive representation}

For $L=(\Z^{5},q)$ in the genus of $L_5$ put $F=2G_q\oplus(-2)$, so that
$\det F=-2^{10}p^{2}$ and $v^{t}Fv=2\bigl(q(x)-z^{2}\bigr)$ for
$v=(x,z)\in\Z^{5}\times\Z$. (The factor $2$ makes $F$ even, as the
implementation requires; it does not change the isotropic vectors.)

\begin{proposition}\label{prop:reduction}
\begin{enumerate}
\item[(a)] $v=(x,z)$ is isotropic for $F$ iff $q(x)=z^{2}$; then $z\ne0$.
\item[(b)] If moreover $\gcd(v)=1$ then $\gcd(x)=1$: $x$ is a witness for $N=|z|$.
\item[(c)] Every totally isotropic subspace of $(\Q^{6},F)$ has dimension at most $1$.
\item[(d)] $F$ has an isotropic vector: every class of the genus has a witness.
\end{enumerate}
\end{proposition}

\begin{proof}
\begin{enumerate}
    \item[(a)]  $q$ is positive definite, so $z=0$ forces $x=0$.
    \item[(b)] $g=\gcd(x)$ gives $g^{2}\mid z^{2}$, so $g\mid\gcd(v)=1$.
    \item[(c)] Over $\mathbb{R}$, $F$ has five positive and one negative square, so a totally
isotropic subspace has dimension at most $1$
\cite[Ch.~IV, \S3.1]{courseinarithmetic}.
    \item[(d)] $F$ is indefinite in six variables: Meyer's theorem
\cite[Ch.~IV, \S3.2, Corollary~2]{courseinarithmetic}.
\end{enumerate}
\end{proof}

Meyer's theorem is not effective. Simon's indefinite LLL \cite{simon} solves $v^{t}Fv=0$ for unimodular $F$; his general algorithm \cite{simondim4} (see \cite{watkins}) reduces an arbitrary $F$, given the factorisation of $\det F$, to that case, and returns a totally isotropic sublattice of maximal dimension, here a single vector, by (c), in time polynomial in the bit size of $F$, the factorisation being the only non-polynomial step \cite[\S8]{simondim4}. Since $\det F=-2^{10}p^{2}$ is known by construction, that step is free at every size of $p$. We use the algorithm as a black box and certify its output ourselves. We note that Proposition~\ref{prop:reduction}(d) asserts only that a witness exists: neither it nor anything else we prove controls the size of the witness that Simon's algorithm returns, and the concentration of $\log_2N$ near $\log_2p$ reported in Section~\ref{data} is an empirical observation.

\begin{algorithm}[ht]
\caption{\texttt{Witness}$(q,p)$: a primitively represented square of one class}
\label{alg:witness}
\Input{$L=(\Z^{5},q)$ in the genus of $L_5$, given by $G_q$; the prime $p$.}
\Output{A primitive $x\in\Z^{5}$ and $N\ge1$ with $q(x)=N^{2}$.}
$F\leftarrow 2G_q\oplus(-2)$\;
$w\leftarrow$\Simon{$F$, $\det F=-2^{10}p^{2}$} \tcp*{one vector, by Proposition~\ref{prop:reduction}(c)}
$v\leftarrow w/\gcd(w)$; write $v=(x,z)$\;
\Verify{$\gcd(v)=1$, $v^{t}Fv=0$, $z\neq0$, $\gcd(x)=1$, $x^{t}G_qx=z^{2}$} in exact integer arithmetic; abort on failure\;
\Return{$(x,|z|)$} \tcp*{a witness, by Proposition~\ref{prop:reduction}(b)}
\end{algorithm}

\subsection{Primitive representation of small square-free integers}
As discussed in Appendix~\ref{RM-Humbert}, the refined Humbert invariant $q_{(\A,\theta)}$ encodes real multiplication by $\mathcal{O}_D$ on a principally polarized abelian surface $(\A,\theta)$: a positive fundamental discriminant $D$ is primitively represented by $q_{(\A,\theta)}$ if and only if there exists an embedding $\mathcal{O}_D\hookrightarrow\End_{\theta}(\A)$, equivalently, if $(\A,\theta)$ lies on the Humbert surface $\mathcal{H}_D$. Computationally, we can directly apply Kannan-Fincke-Pohst enumeration algorithm~\cite{kannan,FinckePohst,SchnorrEuchner}, as in Algorithm~\ref{alg:kzero}, to check if a small square-free integer is primitively represented by a refined Humbert invariant $q_{(\A,\theta)}$ for a given reducible or irreducible polarization.

\begin{algorithm}[ht]
\caption{\texttt{PrimitiveReps}$(G_q,D)$: primitive square-free representations up to $D$}
\label{alg:prim-rm}
\Input{The Gram matrix $G_q$ of a positive definite lattice $L=(\Z^5,q)$; a bound $D\ge2$.}
\Output{The square-free $n$ with $2\le n\le D$ primitively represented by $L$, each with a witness $x\in\Z^5$.}
$S\leftarrow\{2,\ldots,D\}$\;
\For{$d=2,\ldots,\lfloor\sqrt D\rfloor$}{
  delete from $S$ every multiple of $d^2$\;
}
$(G',U)\leftarrow$\LLL{$G_q$}\;
$R\leftarrow\varnothing$ \tcp*{$R$ maps represented targets to witnesses}
\ForEach{$n\in S$}{
  \ForEach{$z$ in \Enumerate{$G',n$}}{
    \If{$\gcd(z)=1$}{
      $x\leftarrow Uz$\;
      \Verify{$x^tG_qx=n$ and $\gcd(x)=1$} in exact integer arithmetic; abort on failure\;
      add $(n,x)$ to $R$\; \textbf{break}\;
    }
  }
}
\Return{$R$}\;
\end{algorithm}

The value $n=1$ is excluded deliberately: by Proposition~\ref{irreducibilitycriteria}, $q_{(\A,\theta)}$ represents $1$ exactly when $\theta$ is reducible, so $n=1$ is a separate test rather than a discarded case. The cost is one enumeration of a ball of radius $\sqrt{D}$ for each surviving $n$; throughout this paper we run Algorithm~\ref{alg:prim-rm} only with $D\le100$.

\section{Applications} \label{applications}
In this section we put the refined Humbert invariant to work. It serves three purposes. First, the
\emph{irreducibility criterion} of Proposition~\ref{irreducibilitycriteria} distinguishes Jacobians
with irreducible principal polarization from product surfaces with reducible polarization. Second,
Kani's splitting theorem (Theorem~\ref{SplittingTheorem}) characterizes $(N,N)$-splittings in terms
of primitive representations of $N^2$ by the refined Humbert invariant $q_{\C}$. Third, primitive
representations of discriminants $D$ detect embeddings of the corresponding real multiplication
orders $\mathcal{O}_D$ into the endomorphism ring. We take these in turn, and in each case work out
what the representation problem says about the isogeny graph: for splittings, a precise relation
between the least level $k_0$ and distance to the split locus, together with the branching that
restricts the walks involved; for real multiplication, a test that needs no walk at all; and for
automorphisms, a comparison against the counts of Ibukiyama--Katsura--Oort. The experiments that
exercise all of this are reported in Section~\ref{data}. Among two-dimensional superspecial isogeny
graphs, Richelot graphs are particularly well suited to explicit computation and already possess a
rich structure, which is why we take them as our main testing ground throughout.

\subsection{Detecting splittings}
Of relevance to isogeny-based cryptography, Costello and Smith's algorithm for the superspecial isogeny problem~\cite{CraigBen} proceeds in two parts. 
The first part performs pseudo-random walks starting from the given  Jacobians to the products of supersingular elliptic curves, obtaining isogenies $\varphi: \mathcal{J(C)} \to E_1 \times E_2$ and  $\varphi' : \mathcal{J(C)'} \to E'_1 \times E'_2 $.
Assuming these walks rapidly approach a uniform distribution, 
it costs $\tilde{O}(p)$ since only an $O(1/p)$ fraction of superspecial abelian surfaces are isomorphic to a product of elliptic curves. The second part uses the Delfs-Galbraith algorithm~\cite{DelfsGalbraith}  to find isogeny paths $E_1 \to E'_1$ and $E_2 \to E'_2$, and then glues them to obtain a path
$\pi : E_1 \times E_2 \to E'_1 \times E'_2 $, costing $\tilde{O}(p^{1/2})$. The final solution is the composition $\Phi := \hat{\varphi}' \circ \pi \circ \varphi$, and the expected running time is $\tilde{O}(p)$, dominated by the first phase.
Motivated by the isogeny-based cryptography discussions above, we aim to give a new method to detect close split surfaces to a randomly chosen vertex in the isogeny graph. It does not depend on any isogeny calculation; it does not have to be on the Richelot isogeny graph, and it gives us a more conceptually general method in theory for any $(N, N)$ split case. On the other hand, as a fallback, we do not know the geometric description of the principally polarized superspecial surface.

We now continue the analysis of this problem with refined Humbert invariants. Let $q_{\C}$ be a refined Humbert invariant of a principally polarized superspecial Jacobian abelian surface. We can always guarantee choosing a random polarization in a way that it leads to a Jacobian variety; this can be done efficiently by computing the refined Humbert invariant $q_{(\A,\theta)}$ by calculations in \cite{edagaurish}, and verifying that it is actually a Jacobian by the irreducibility criterion, Proposition~\ref{irreducibilitycriteria}. We emphasize that we do not know which abelian surface $\mathcal{J(C)}$ is exactly, so it is just a random vertex corresponding to a $(\mathcal{J(C)},\theta_\C)$. By Kani's splitting theorem, see Theorem~\ref{SplittingTheorem}, we have that a primitive representation of $4^k$ by the associated refined Humbert invariant $q_{\C}$ gives rise to a $(2^k,2^k)$-isogeny from $\mathcal{J(C)}$ to a split surface $E \times E'$. The least such exponent is exactly the quantity $k_0$ of Section~\ref{designing}, and it measures a
restricted notion of distance to the split locus, which we now make precise. Write
$\mathcal{S}\subset\Gamma^{SS}_2(2;p)$ for the set of vertices carrying a reducible polarization, and
$d(\cdot,\cdot)$ for the graph distance in $\Gamma^{SS}_2(2;p)$.

\paragraph{Branching for splittings}
The isogeny walks based on Kani's splitting theorem, Theorem~\ref{SplittingTheorem}, are not arbitrary random walks in the isogeny graph. Recall that by Theorem~\ref{SplittingTheorem}, a primitive representation $q_{\C}(x)=N^2$
corresponds to an $(N,N)$-splitting of $\mathcal{J}(C)$. When an $(N,N)$ split isogeny is a composition of lower-degree isogenies, primitiveness implies that every intermediate vertex is a Jacobian surface. Thus, we consider isogeny paths of the form
\[
\mathcal{J(C)}\rightarrow  \mathcal{J}(\C_1)\rightarrow \cdots
\rightarrow \mathcal{J}(\C_k)\rightarrow E_1\times E_2,
\]
where no intermediate vertex $\mathcal{J}(\C_i)$ is a split surface. In particular, the split surface can only be reached at the final vertex. Hence, the isogeny walks here are a restricted class of isogeny walks.

Moreover, when $N=\ell^k$ and the corresponding splitting is factored into $(\ell,\ell)$-isogenies, primitiveness implies that the kernel of every intermediate isogeny is isomorphic to $(\Z/\ell^i\Z)^2$. Consequently, each isogeny after the first is a good extension of the preceding one in the sense of \cite{CastryckDecruMultiradical}. Equivalently, if $\A_{i-1}\to\A_i$ is fixed, then the kernel of the next isogeny $\A_i\to\A_{i+1}$ intersects the kernel of the dual isogeny $\A_i\to\A_{i-1}$ trivially.
This condition gives a uniform reduction in the number of possible continuations of the walk. The full $(\ell,\ell)$-isogeny graph has weighted out degree $d_\ell=(\ell+1)(\ell^2+1)$, so there are $d_\ell^k$ unrestricted isogeny walks of length $k$. The first isogeny has $d_\ell$ possible kernels. Once an incoming isogeny has been fixed, exactly $\ell^3$ of the $d_\ell$ outgoing $(\ell,\ell)$-isogenies are good extensions; see \cite[Lemma~2.1]{Genus2hashfunction}, and which is also the branching used in the good-extension digraph of \cite{DecruReijnders}. Hence the number of isogeny walks of length $k$ in which every isogeny after the first isogeny is a good extension is $\ell^{3(k-1)}d_\ell$.

Therefore, among all unrestricted isogeny walks of length $k$, the proportion satisfying the good-extension condition is $\left(\frac{\ell^3}{d_\ell}\right)^{k-1}=\left(\frac{\ell^3}{(\ell+1)(\ell^2+1)}\right)^{k-1}$. For the case of $\ell=2$, this proportion is $\left(\frac{8}{15}\right)^{k-1}$.
The walks associated with splittings derived from refined Humbert invariants $q_{\C}$ satisfy an additional restriction: every intermediate vertex must remain a Jacobian, while a split surface may occur only at the final vertex. Thus $\ell^3$ is the number of good extensions available after a fixed incoming isogeny, but not every such extension is admissible for a primitive splitting. At an intermediate step, only good extensions with Jacobian codomain are allowed, whereas at the final step the codomain must be a product surface. Hence the actual branching for the splitting walks considered here is bounded above by $\ell^3$ at every step after the first, and may be strictly smaller because of the requirement that the walk remain in the Jacobian locus until its final edge.

\begin{lemma}\label{lem:k0-distance}
Let $(\A,\theta)$ be a principally polarized superspecial abelian surface with irreducible
polarization and refined Humbert invariant $q_{(\A,\theta)}$, and let $k_0=k_0(\A,\theta)$ be the
least $k\geq1$ such that $q_{(\A,\theta)}$ primitively represents $4^{k}$. Then
\begin{enumerate}
\item[(i)] $k_0$ is the least length of a walk from $(\A,\theta)$ to $\mathcal{S}$ in
$\Gamma^{SS}_2(2;p)$ in which every edge after the first is a good extension of its predecessor and
no intermediate vertex lies in $\mathcal{S}$;
\item[(ii)] consequently $d\bigl((\A,\theta),\mathcal{S}\bigr)\leq k_0$, with equality if and only if
some geodesic from $(\A,\theta)$ to $\mathcal{S}$ is a chain of good extensions; and
\item[(iii)] equality holds in (ii) whenever $k_0\leq2$.
\end{enumerate}
\end{lemma}

\begin{proof}
(i) By Theorem~\ref{SplittingTheorem}, a primitive representation of $4^{k}$ by $q_{(\A,\theta)}$ is
equivalent to a $(2^{k},2^{k})$-isogeny $\varphi:\A\to E_1\times E_2$, that is, one whose kernel is
maximal isotropic in $\A[2^{k}]$ and hence isomorphic to $(\Z/2^{k}\Z)^2$. Factoring $\varphi$ into
$k$ successive $(2,2)$-isogenies, the condition on $\ker\varphi$ says precisely that each step after
the first is a good extension of the preceding one, and primitivity says that no intermediate
codomain lies in $\mathcal{S}$; both points are the content of the branching discussion above.
Conversely, a walk with these two properties composes to an isogeny with kernel
$(\Z/2^{k}\Z)^2$, hence to a $(2^{k},2^{k})$-splitting, hence to a primitive representation of
$4^{k}$. Minimality of $k_0$ now gives the claim.

(ii) The walk furnished by (i) has length $k_0$ and ends in $\mathcal{S}$, so
$d((\A,\theta),\mathcal{S})\leq k_0$. If some geodesic is a chain of good extensions then its length
is at least $k_0$ by the minimality in (i), and at most $k_0$ because it is a geodesic; so the two
agree. Conversely, if $d((\A,\theta),\mathcal{S})=k_0$ then the walk furnished by (i) has length
exactly $d((\A,\theta),\mathcal{S})$, so it is itself a geodesic --- and it is a chain of good extensions.

(iii) By Theorem~\ref{SplittingTheorem}, $q_{(\A,\theta)}$ primitively represents $4$ precisely when
$\A$ is $(2,2)$-split, that is, precisely when $\A$ has a neighbor in $\mathcal{S}$; so $k_0=1$ if
and only if $d((\A,\theta),\mathcal{S})=1$. Since $\theta$ is irreducible,
$(\A,\theta)\notin\mathcal{S}$ and hence $d((\A,\theta),\mathcal{S})\geq1$. So if $k_0=2$ then
$d((\A,\theta),\mathcal{S})\neq1$, while $d((\A,\theta),\mathcal{S})\leq2$ by (ii); the two
therefore agree.
\end{proof}

Whether the inequality in (ii) can be strict we leave open: by (iii) any witness must have $k_0\geq3$, and deciding the question requires the graph distance $d$, which the method of this paper does not compute.

Thus $k_0$ is an upper bound for the graph distance to the split locus, and is exactly the distance
measured along walks that correspond to a single $(2^{k},2^{k})$-isogeny. The latter is the
quantity of cryptographic interest, since it is $(N,N)$-isogenies, and not arbitrary walks, that
one computes. In fact, the set of all primitively represented $4^k$ for exponents $k$ describes the
set of split surfaces reachable from $\mathcal{J(C)}$ by isogenies of degree $(2^k,2^k)$. Therefore, we restrict our attention to values of $k$ up to a diameter bound $k_{\mathrm{diameter}}$ for the isogeny graph of principally polarized superspecial abelian surfaces. By the discussion in Section~\ref{sec:isogenygraphs}, the Chung--Faber--Manteuffel estimate gives $k_{\mathrm{diameter}}=4.20\log_2(p)+O(1)$ unconditionally, and $k_{\mathrm{diameter}}=3\log_2 (p)+O(1)$ under \cite[Conjecture~4.10]{floritsmith};  for any fixed $p$ a rigorous bound is obtained by substituting the exact value $M=\#\Gamma^{SS}_2(2;p)$ into that inequality. For the small primes of \cite[Appendix~A]{floritsmith} the diameter has been computed exactly, see Table~\ref{tab:richelot_diameter_p_11mod12}, and we use the tabulated value in place of the bound whenever it is available. 

Our data is obtained differently from that of \cite{floritsmith}: they construct the full Richelot isogeny graph, computing the surfaces via Igusa--Clebsch invariants and the edges via compact Richelot isogeny formulas, whereas we read the same information off the refined Humbert invariant. 

\begin{remark}
    The difference between our method and the results in \cite{Splitting} is one of input: we use neither the equations of the genus $2$ curve, which would be obtained by computing Igusa invariants, nor an explicit chain of isogenies. We emphasize that we simply chose a random Jacobian in the isogeny graph, and that we do not know the equations of the curve $\C$ of genus $2$. We make no claim of a runtime advantage: we have not benchmarked the two approaches against one another, and any such comparison would depend on the regime of $N$, since \cite{Splitting} targets prescribed $N\leq11$ whereas $N$ is determined intrinsically here.
\end{remark}

\begin{remark}
It is an interesting question whether one can compute the refined Humbert invariant $q_{\C}$ for a given model of curve $\C$ of genus $2$.
The question of finding an algorithm to determine $q_{\C}$ for a given curve $\C$ is mostly open. Note that there is no such algorithm to compute the quintic refined Humbert invariant $q_{\C}$ for a given model of $\C$. The existence of such an algorithm would be a powerful method for isogeny graphs and isogeny problems, yet it is challenging to solve in the general case. Evidence that this difficulty is intrinsic is given by K{\i}r{\i}ml{\i}--Martindale~\cite{kirimlimartindale}, who show that in the settings of SQIsign and CSIDH, and under GRH, the computational isogeny problem is polynomially equivalent to the computational refined Humbert invariant problem; that work also surveys what is known about computing refined Humbert invariants.
\end{remark}

\subsection{Detecting endomorphisms}
For supersingular elliptic curves, computing endomorphisms is essentially as hard as finding isogenies ~\cite{wesolowski2021supersingularisogenypathendomorphism};
Wesolowski showed a GRH-based polynomial-time reduction between the isogeny problem over $\F_{p^2}$ and the endomorphism ring problem. Further, Page--Wesolowski showed that finding a single nontrivial $\varphi \in \End(E)$ is equivalent to solving the supersingular isogeny problem~\cite{welowoski-page}.
This naturally raises the question of whether a similar equivalence holds in dimension~$2$, yet there is neither a widely accepted algorithm nor a mature complexity analysis for computing endomorphism rings of superspecial principally polarized abelian surfaces. 
We now show that the refined Humbert invariant gives a direct arithmetic test for real multiplication, and use it to estimate how common RM is among superspecial surfaces. Recall from Appendix~\ref{RM-Humbert} that every $\mathcal{H}(q)$ is contained in any Humbert surface $\mathcal{H}_N$ that $q$ primitively represents, in arbitrary characteristic by \cite[Definition 8, Proposition 12]{KaniGeneralizedHumbert}.

We investigate RM orders experimentally using the complete sets of principal polarizations for small primes up to $191$, and, separately, by sampling $100{,}000$ random principal polarizations for each prime $227\leq p\leq1619$ and determining the primitively represented square-free discriminants $D\leq100$.
\subsubsection{Frequency of real multiplication with refined
Humbert invariants $q_{(\A,\theta)}$}
We describe an experiment to estimate the frequency of principally polarized superspecial abelian surfaces admitting real multiplication by a prescribed quadratic order $\mathcal{O}_D$. 
For each sampled principally polarized superspecial abelian surface $v_i=(\A,\theta_i)$ over $\overline{\mathbb{F}}_p$, we compute its quintic refined Humbert invariant $q_i=q_{(\A,\theta_i)}$.

Fix a positive discriminant $D$ corresponding to a real quadratic order $\mathcal{O}_D$. We test whether $q_i$ primitively represents $D$, namely, whether there exists a primitive vector $x$ such that $q_i(x)=D$.
By the primitive-representation criterion, see Theorem~\ref{RMtheorem}, this detects an embedding of the corresponding real quadratic order $\mathcal{O}_D$ into the additive subgroup $\End_{\theta_i}(\A)$ consisting of endomorphisms fixed by the Rosati involution with respect to $\theta_i$.

Thus, RM detection is carried out directly from $q_i$, without requiring genus $2$ curve equations, Igusa invariants, explicit equations for describing the Humbert surfaces, or the computation of Richelot isogenies or  $(\ell,\ell)$-isogenies.

A further advantage is that a single sampled refined Humbert invariant $q_i$ can be tested simultaneously for many discriminants $D$. Once $q_i$ is computed, we determine the set of positive discriminants $D\leq D_{\max}$ that it primitively represents. The same collection of sampled principally polarized abelian surfaces therefore yields estimates for all discriminants in the prescribed range and allows the resulting data to be grouped according to arithmetic invariants of $\mathcal{O}_D$, such as its class number or narrow class number. We further emphasize that our method does not require the class number or narrow class number of $\mathcal{O}_D$ to be equal to $1$.
 
Motivated by the square-discriminant case \cite[Section~6.2]{floritsmith}, the following heuristic on the distribution of RM vertices is extended in \cite[Section~8.2]{detectionofendomorphisms} to an arbitrary discriminant $D$ by introducing a function $\widetilde{\mathfrak{D}}_D$. The expected number of steps required for a random walk in $\Gamma_2^{SS}(2;p)$ to reach an abelian surface admitting RM by $\mathcal{O}_D$ is then predicted to be
\begin{equation}\label{RMheuristic}
    \frac{p}{\widetilde{\mathfrak{D}}_D}+O(1),
\end{equation}
where $\widetilde{\mathfrak{D}}_D=O(D^{3/2})$ and, for square discriminants $D=N^2$, it satisfies
$\widetilde{\mathfrak{D}}_{N^2}=5\mathfrak{D}_N$ with $\mathfrak{D}_N$ denoting the number of maximal isotropic $(N,N)$-subgroups of the $N$-torsion of a principally polarized abelian surface. A heuristic for $\widetilde{\mathfrak{D}}_D$ is given in \cite{detectionofendomorphisms} as follows.
Let $D$ be a fundamental discriminant such that $\mathcal{O}_D$ has narrow class number one, and let $K=\operatorname{Frac}(\mathcal{O}_D)$. As explained in \cite[\S 8.2.1]{detectionofendomorphisms}, combining \cite[Theorem~2.5.35]{Nicole2005} with the Eichler mass formula \cite[Theorem~26.1.5]{voight} gives the estimate
$\frac{\zeta_K(2)D^{3/2}}{8\pi^4}p^2+O(p)$
for the number of superspecial principally polarized abelian surfaces equipped with an embedding
$\mathcal{O}_D\hookrightarrow\End(\A)$, with the discrepancy from the unweighted count governed by the Eichler class number formula. For $D\ll p$, it is further expected that almost every surface admitting RM by $\mathcal{O}_D$ gives rise to exactly two such embeddings. This leads to the following heuristic $\widetilde{\mathfrak{D}}_D
=\frac{1}{2}\frac{360\zeta_K(2)}{\pi^4}D^{3/2}.
$

\begin{remark}  
In contrast to the random-walk experiment of~\cite{detectionofendomorphisms}, we estimate the occurrence of real multiplication directly from refined Humbert invariants. Rather than starting from a fixed vertex of the superspecial $(2,2)$-isogeny graph and measuring the first hitting time of the locus admitting RM by $\mathcal{O}_D$, we sample principally polarized superspecial abelian surfaces $(\A,\theta)$, compute their refined Humbert invariants $q_{(\A,\theta)}$, and test whether $D$ is primitively represented. This directly estimates the density $\delta_D(p)=\Pr(q_{(\A,\theta)}\text{ primitively represents }D)$ of vertices admitting RM by $\mathcal{O}_D$. In~\cite{detectionofendomorphisms}, $\mu_D$ denotes the mean number of steps in $\Gamma_2^{SS}(2;p)$ required to encounter a principally polarized abelian surface with RM by $\mathcal{O}_D$. Under the heuristic that a sufficiently mixed Richelot walk samples from the same distribution, one expects $\mu_D\simeq \delta_D(p)^{-1}$ and hence $p/\mu_D\simeq p\,\delta_D(p)$. Our approach neither uses explicit curve equations nor a random walk in the Richelot isogeny graph; instead, the main issue is to ensure that the sampled polarizations induce the desired distribution on principally polarized isomorphism classes.
\end{remark}

\subsection{Automorphism groups}

Computing refined Humbert invariants detects genus $2$ Jacobians with extra automorphisms and thereby identifies special vertices and their local neighborhoods in the isogeny graph. This viewpoint was investigated experimentally for Richelot graphs in \cite{floritsmith} using reduced automorphism groups. By Proposition~\ref{auto-rhi-relation}, the refined Humbert invariants determine the automorphism group of the corresponding principally polarized abelian surface. 

Here, we illustrate some experiments verifying the relationship between the automorphism groups of curves $\C$ of genus 2 and their associated refined Humbert invariants $q_{\C}$. Note that the list of possibilities for $\Aut(\C)$ for any genus $2$ curve is given in the first row of Table~\ref{Tab:Autr4}, see~\cite[Theorem 2]{shaska2004elliptic}, \cite{bolza}. Recall from Proposition~\ref{auto-rhi-relation} that if $\iota(\Aut(\C))$ denotes the number of involutions of the group $\Aut(\C)$, then we have that $\iota(\Aut(\C))-1=r_4(q_{\C})$ by~\cite[Proposition 26]{KaniCurvesAbelianSurfaces}. 
Although the statement of Theorem~\ref{Ibuyikama-auto-count} is given by reduced automorphism groups $\mathrm{RA(\C)}$, we converted them back to $\Aut(\C)$ in Table~\ref{auto-experiment}.

The polarizations used here are not sampled: we take them from the exhaustive enumeration of
\cite[Algorithm~2]{edagaurish}, which as recalled in Section~\ref{designing} is complete for
$p\leq191$ and, at $p=227,239,251$, misses only reducible polarizations, so that the irreducible
ones --- the case of interest here --- are complete throughout the range. For the primes $p \in \{11,\allowbreak
23,\allowbreak
47,\allowbreak
59,\allowbreak
71,\allowbreak
83,\allowbreak
107,\allowbreak
131,\allowbreak
167,\allowbreak
179,\allowbreak 191,\allowbreak 227,\allowbreak239,\allowbreak251\}$ we compute the refined Humbert invariant of every polarization, discard those that represent $1$,
and classify the remainder up to isometry. These are the refined Humbert invariants $q_{\C}$
corresponding to principally polarized superspecial surfaces (with irreducible polarizations)
$(\A,\theta)=(\mathcal{J(C)},\theta_{\C})$.

\begin{table}[!ht]
    \caption{Automorphism data for the principal polarizations of $\A=E\times E$, enumerated with
    \cite[Algorithm 2]{edagaurish}. The count columns are: $\#\Theta=\mathbf{H}(p)$, the principal
    polarizations up to isomorphism; $\#\mathrm{RHI}_{\mathrm{irred}}=\mathbf{H}-\frac{\mathbf{h}(\mathbf{h}+1)}{2}$,
    those that are irreducible; $\#\mathrm{RHI}_{\mathrm{iso}}$, the number of distinct refined Humbert
    invariants they realize up to isometry; and $\#\Gen{q_5}_{\text{not }1}$, the number of classes in the
    genus of $q_5$ that do not represent $1$. The last two columns agree in every row, so the irreducible
    polarizations realize \emph{every} class of the genus that does not represent $1$. Within each pair,
    the first entry is the number of \emph{polarizations} predicted by Theorem~\ref{Ibuyikama-auto-count} (the
    column sums to $\#\mathrm{RHI}_{\mathrm{irred}}$), and the second the number of \emph{isometry classes}
    observed (summing to $\#\mathrm{RHI}_{\mathrm{iso}}$); these count different objects, so they coincide
    only where little collapse occurs. The $C_{10}$ entries are marked by $*$ because $r_4$ cannot separate
    $C_2$ from $C_{10}$; see the discussion after the table.}
    \label{auto-experiment}
    \centering
\begin{adjustbox}{max width=\textwidth}
\begin{tabular}{cccccccccccc}
    \toprule
    \multirow{2}{*}{$p$} & \multirow{2}{*}{$\#\Theta$} & \multirow{2}{*}{$\#\underset{\text{irred}}{\text{RHI}}$} & \multirow{2}{*}{$\#\underset{\text{iso}}{\text{RHI}}$} & \multirow{2}{*}{$\#\underset{\text{not 1}}{\Gen{q_5}}$} & \multicolumn{7}{c}{$\#\Aut(\C)$ type: (\texttt{Theorem~\ref{Ibuyikama-auto-count} bound}, \texttt{Table~\ref{Tab:Autr4} count})} \\
    \cmidrule(lr){6-12} 
    & & & & & $\text{C}_2$ & $\text{C}_{10}$ &$\text{C}_2\times \text{C}_2$ & $\text{D}_4$ & $\text{D}_6$ & $\text{C}_3\rtimes \text{D}_4$ & $\GL_2(\mathbb{F}_3)$\\ 
    \midrule
    11 & 5 & 2 & 2 & 2 & $(0, 0)$ & $(0, 0)$ & $(0,0)$ & $(0,0)$ & $(1,1)$ & $(1,1)$ & $(0,0)$ \\ 
    23 & 16 & 10 & 9 & 9 & $(0,0)$ & $(0,0)$ & $(5,4)$ & $(1,1)$ & $(2,2)$ & $(1,1)$ & $(1,1)$ \\
    47 & 72 & 57 & 43 & 43  & $(14,9)$ & $(0,0)$ & $(31,23)$ & $(4,4)$ & $(6,5)$ & $(1,1)$ & $(1,1)$ \\
    59 & 125 & 104 & 74 & 74 & $(35,25)$ & $(1,1^*)$ & $(52,35)$ & $(6,5)$ & $(9,7)$ & $(1,1)$ & $(0,0)$ \\
    71 & 198  & 170 & 116 & 116 & $(70,42)$ & $(0,0)$ & $(81,58)$ & $(7,6)$ & $(10,8)$ & $(1,1)$ & $(1,1)$ \\
    83 & 296 & 260 & 165 & 165 &  $(123,80)$ & $(0,0)$ & $(114,67)$ & $(9,7)$ & $(13,10)$ & $(1,1)$ & $(0,0)$ \\
    107 & 581 & 526 & 311 & 311 & $(296, 178)$ & $(0,0)$ & $(200,111)$ & $(12,8)$ & $(17,13)$ & $(1,1)$ & $(0,0)$\\ 
    131 & 1008  & 930  & 554 & 554 & $(583,340)$ & $(0,0)$ & $(310,186)$ & $(15,11)$ & $(21,16)$ & $(1,1)$ & $(0,0)$ \\ 
    167 & 1978  & 1858 & 1058 & 1058 & $(1290,709)$ & $(0,0)$ & $(521,313)$ & $(19,15)$ & $(26,19)$ & $(1,1)$ & $(1,1)$ \\ 
    179 & 2404 & 2268 & 1266 & 1266 & $(1614,894)$ & $(1,1^*)$ & $(602,336)$ & $(21,14)$ & $(29,20)$ & $(1,1)$ & $(0,0)$ \\
    191 & 2886 & 2733 & 1540 & 1540 & $(1988,1079)$ & $(0,0)$ & $(691,420)$ & $(22,17)$ & $(30,22)$ & $(1,1)$ & $(1,1)$ \\ 
    227 & 4712 & 4502 & 2425 & 2425 & $(3447,1848)$ & $(0,0)$ & $(990,533)$ & $(27,18)$ & $(37,25)$ & $(1,1)$ & $(0,0)$ \\ 
    239 & 5460  & 5229 & 2869 & 2869 & $(4057,2164)$ & $(1,1^*)$ & $(1103,654)$ & $(28,21)$ & $(38,27)$ & $(1,1)$ & $(1,1)$  \\ 
    251 & 6281 & 6028 & 3271 & 3271  & $(4736,2541)$ & $(0,0)$ & $(1220,680)$ & $(30,21)$ & $(41,28)$ & $(1,1)$ & $(0,0)$ \\ 
    \bottomrule
    \end{tabular}
  \end{adjustbox}
\end{table} 

Two comparisons in Table~\ref{auto-experiment} are meaningful, and one is not. First, each column sums to
the total it should: the predicted counts sum to $\#\mathrm{RHI}_{\mathrm{irred}}$ and the observed ones
to $\#\mathrm{RHI}_{\mathrm{iso}}$, in every row. Second, for the rare types $D_4$, $D_6$,
$C_3\rtimes D_4$ and $\GL_2(\F_3)$ the two entries of a pair are equal or nearly so, because almost no
collapse occurs there; in particular the unique $C_3\rtimes D_4$ curve, and the unique $\GL_2(\F_3)$
curve whenever Theorem~\ref{Ibuyikama-auto-count} predicts one, are recovered at every prime tested. By
contrast, no agreement should be expected in the $C_2$ and $C_2\times C_2$ columns: at $p=251$ the $C_2$
entry falls from $4736$ predicted polarizations to $2541$ observed isometry classes, which measures the
collapse and not a discrepancy, since the two numbers count different objects. A further observation is
that $\#\mathrm{RHI}_{\mathrm{iso}}$ equals $\#\Gen{q_5}_{\text{not }1}$ in every row: the irreducible
polarizations of $E\times E$ realize the full genus of $q_5$ away from the classes representing $1$. The entries marked with a star require a caveat. By Table~\ref{Tab:Autr4} we have $r_4(q_{\C})=0$ for both $\Aut(\C)\cong C_2$ and $\Aut(\C)\cong C_{10}$, so the refined Humbert invariant cannot distinguish these two cases; the $C_{10}$ entries are therefore not an independent measurement, but the count of Theorem~\ref{Ibuyikama-auto-count} subtracted from the total number of classes with $r_4=0$. More generally, the counts we obtain are smaller than the predicted ones because two different principally polarized surfaces $(\mathcal{J(C)},\theta_{\C})$ and $(\mathcal{J(C')},\theta_{\C'})$ can share a refined Humbert invariant $q_{\C}\sim q_{\C'}$ up to equivalence. 

An experiment of this kind was carried out in \cite{floritsmith}, where they construct Richelot isogeny graphs and compute chains of isogenies and Igusa-Clebsch invariants. In contrast, our approach does not require the explicit construction or navigation of the full isogeny graph of principally polarized superspecial abelian surfaces, but we do not know which automorphism group belongs to which principally polarized superspecial surface. Instead, we compute the associated refined Humbert invariants directly, thereby extracting the necessary information without enumerating the vertices and edges of the underlying graph. Moreover, the proposed method is generic in the sense that it is independent of any auxiliary choice of a particular $(\ell,\ell)$-isogeny graph.  Given that the set of polarizations is complete, we can also count curves  $\C$ of genus 2 with specified $\Aut(\C)$ just by checking the representation of $4$ by the corresponding quintic refined Humbert invariant $q_{\C}$.

\section{Data} \label{data}
\subsection{The experiment with all the polarizations for small primes}
\paragraph{Splittings} We start with small characteristic primes $p$ for which all principal polarizations are known, using the verified polarizations from \cite{edagaurish}. This enables us to analyze the full splitting behavior of $(2^k,2^k)$-isogeny walks in the superspecial Richelot isogeny graph $\Gamma_2^{SS}(2;p)$ and of $(3^k,3^k)$-isogeny walks in the superspecial $(3,3)$-isogeny graph $\Gamma_2^{SS}(3;p)$, for primes $11\leq p\leq 191$. The same experiment can be carried out efficiently for other small primes $\ell$, particularly for $\ell<31$.

In the first experiment in Figure~\ref{fig:splitting-experiments}(a), we record, for each refined Humbert invariant in the sample and for each $\ell\in\{2,3\}$, the smallest exponent $k_{\min}$ for which $\ell^{2k}$ is primitively represented. The resulting distributions are shown as functions of $p$. For the smaller primes, primitive representation frequently occurs already at $k=1$, whereas for larger $p$ the mass gradually shifts toward $k_{\min}=2$ and $k_{\min}=3$. This shift is more pronounced for $\ell=2$: throughout the larger-prime range, the $\ell=3$ experiment places a substantially larger proportion of the forms at $k_{\min}=1$ or $2$. Thus, although the smallest represented exponent tends to increase with $p$, the computations indicate that primitive representations of powers of $3$ are typically attained at smaller exponents than the corresponding powers of $2$. This is what the branching count of Section~\ref{applications} predicts: the $(\ell,\ell)$-isogeny graph has weighted degree $d_\ell=(\ell+1)(\ell^2+1)$, so $d_2=15$ against $d_3=40$, and the number of good extensions available after a fixed incoming isogeny is $\ell^3$, so $8$ against $27$. Balls around a vertex therefore grow considerably faster at $\ell=3$, and the split locus is met at a smaller exponent, which is precisely the shift seen between the two panels.

The second experiment in Figure~\ref{fig:splitting-experiments}(b) measures the extent to which each refined Humbert invariant represents the entire range of tested powers, rather than only the first one that occurs. For each form, we count the tested exponents $k$ for which $\ell^{2k}$ fails to be primitively represented, with a separate class for forms representing every tested value $1\leq k\leq K_{\max}$; here $K_{\max}=\lceil\ln p\rceil$ as in Section~\ref{sec:primreps}, which is $3$ at $p=11$ and $6$ at $p=251$. The proportion representing all tested exponents decreases as $p$ grows, particularly for $\ell=2$, while the corresponding proportion for $\ell=3$ remains consistently larger. Nevertheless, the overwhelming majority of forms represent all but at most one or two of the tested exponents. This experiment, therefore, complements the $k_{\min}$-distribution by showing that a small first represented exponent does not by itself determine the subsequent representation pattern, and by quantifying how close the observed forms are to representing the full tested range.

\begin{figure}[hbt]
    \centering

    \begin{minipage}[t]{0.49\textwidth}
        \centering
        \includegraphics[width=\linewidth]{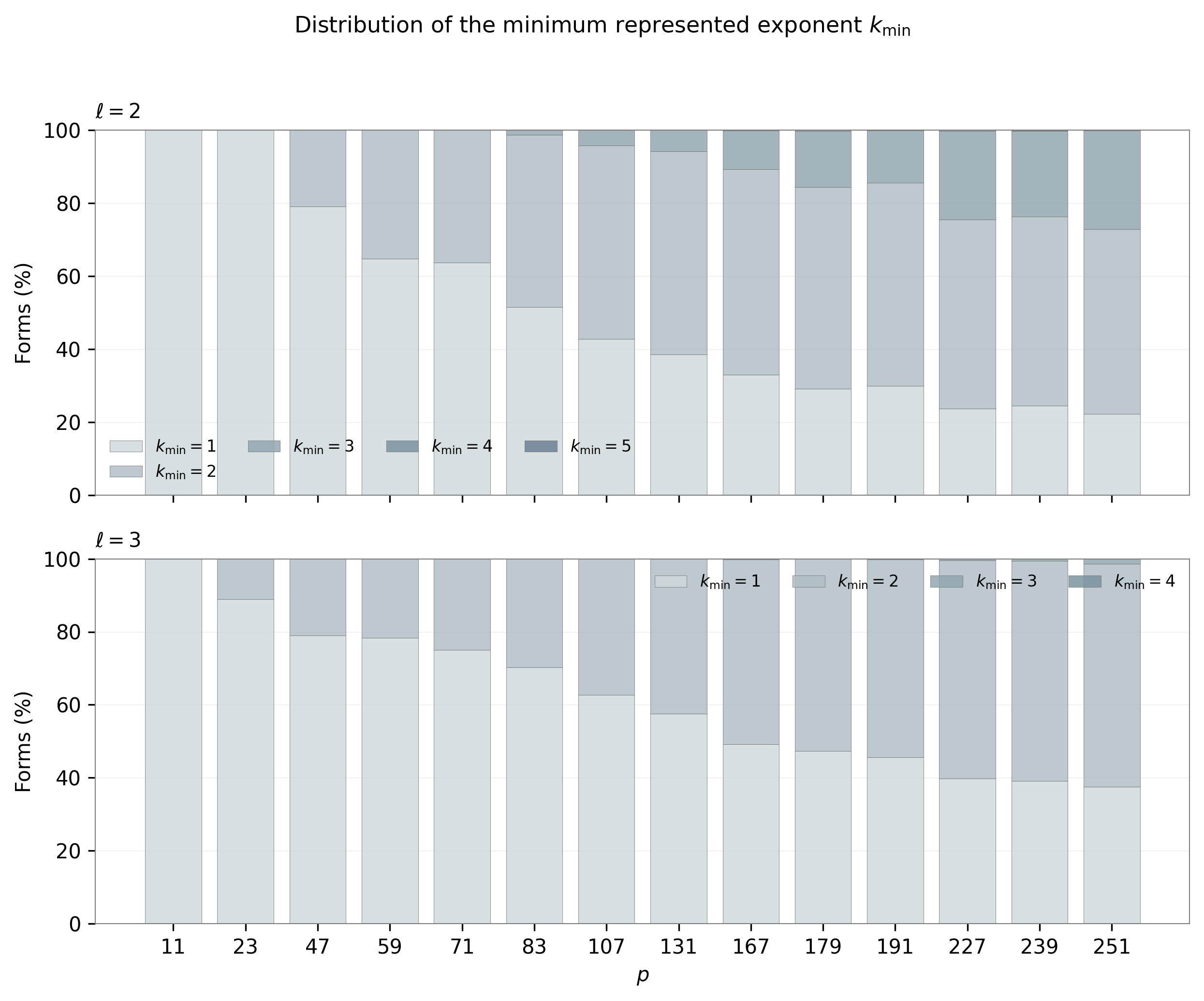}

        \smallskip
        \textbf{(a)}
    \end{minipage}
    \hfill
    \begin{minipage}[t]{0.49\textwidth}
        \centering
        \includegraphics[width=\linewidth]{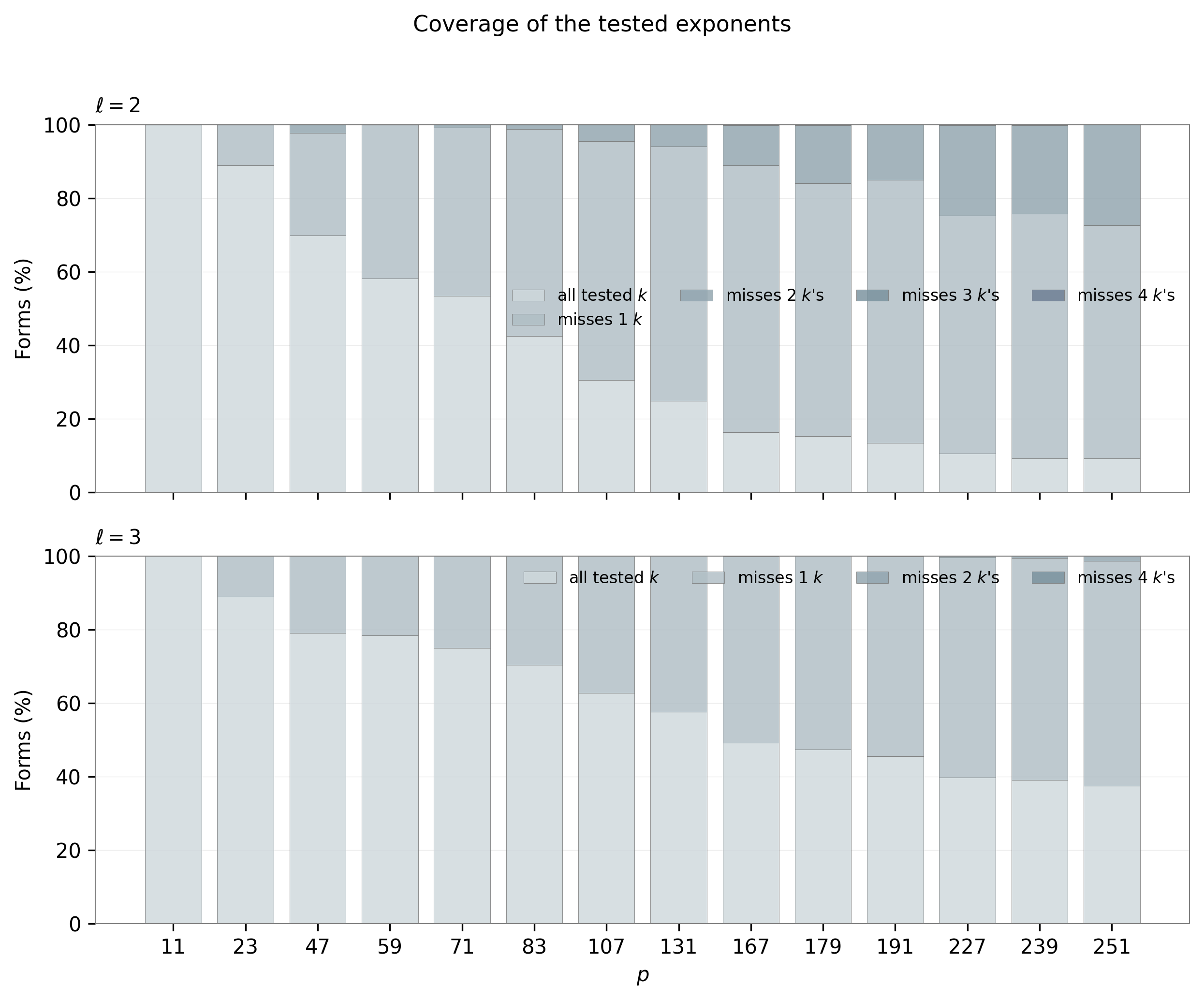}

        \smallskip
        \textbf{(b)}
    \end{minipage}

    \caption{%
    \textbf{(a)} For each prime $p$, the stacked bars give the proportion of
    refined Humbert invariants for which $k_{\min}=k$, where $k_{\min}$ is the
    smallest tested exponent such that $\ell^{2k}$ is primitively represented.
    The upper and lower panels correspond to $\ell=2$ and $\ell=3$,
    respectively. As $p$ increases, the distribution shifts from
    $k_{\min}=1$ toward larger values, with the shift occurring more slowly
    for $\ell=3$ than for $\ell=2$.\\
    \textbf{(b)} For each prime $p$, the stacked bars record the proportion of
    refined Humbert invariants according to the number of tested exponents
    $k$ for which $\ell^{2k}$ is not primitively represented. The segment
    labelled ``all tested $k$'' consists of forms representing $\ell^{2k}$
    primitively for every $1\leq k\leq K_{\max}$; the remaining segments
    distinguish forms missing one or more tested values of $k$. The upper and
    lower panels correspond to $\ell=2$ and $\ell=3$, respectively, where
    $K_{\max}=\lceil\ln p\rceil$ is the maximal exponent tested for the corresponding prime $p$.}

    \label{fig:splitting-experiments}
\end{figure}

\paragraph{RM vertices} We start with small characteristic primes $p$ for which all principal polarizations are known,
using the verified polarizations from \cite{edagaurish}. We report how RM by $\mathcal{O}_D$ appears for the full isogeny graphs for these primes.
By a \emph{Type} we mean an isometry class of refined Humbert invariants occurring in the sample. Distinct principally polarized surfaces may share a Type, so each Type accounts for at least one, and possibly several, vertices of the isogeny graph.
For each prime $p$ and each Type, we consider the primitively represented square-free values, excluding $D=1$, and define
$D_{\min}:=\min\{D>1 : D \text{ is primitively represented}\}$.

Figure~\ref{fig:min-D-summary}(a) shows the distribution of $D_{\min}$ across the Types for each $p$, with marker area proportional to the corresponding proportion. The minima are concentrated mainly at $ D=5, 13,$ and $17$, while larger values occur with substantially lower frequency. The second figure considers the cumulative proportion $F_p(D_0):=
\frac{\#\{T : D_{\min}(T)\leq D_0\}}{\#\{T\}}$ for $D_0\in\{5,13,17,21,29\}$. Thus, $F_p(D_0)$ measures the proportion of Types for which a nontrivial primitive representation occurs by the threshold $D_0$. Although the proportions for the smallest thresholds decrease as $p$ grows, the thresholds $17$, $21$, and $29$ still account for a large proportion of the Types.

\begin{figure}[ht!]
    \centering

    \begin{minipage}[t]{0.49\textwidth}
        \centering
        \includegraphics[width=\linewidth]{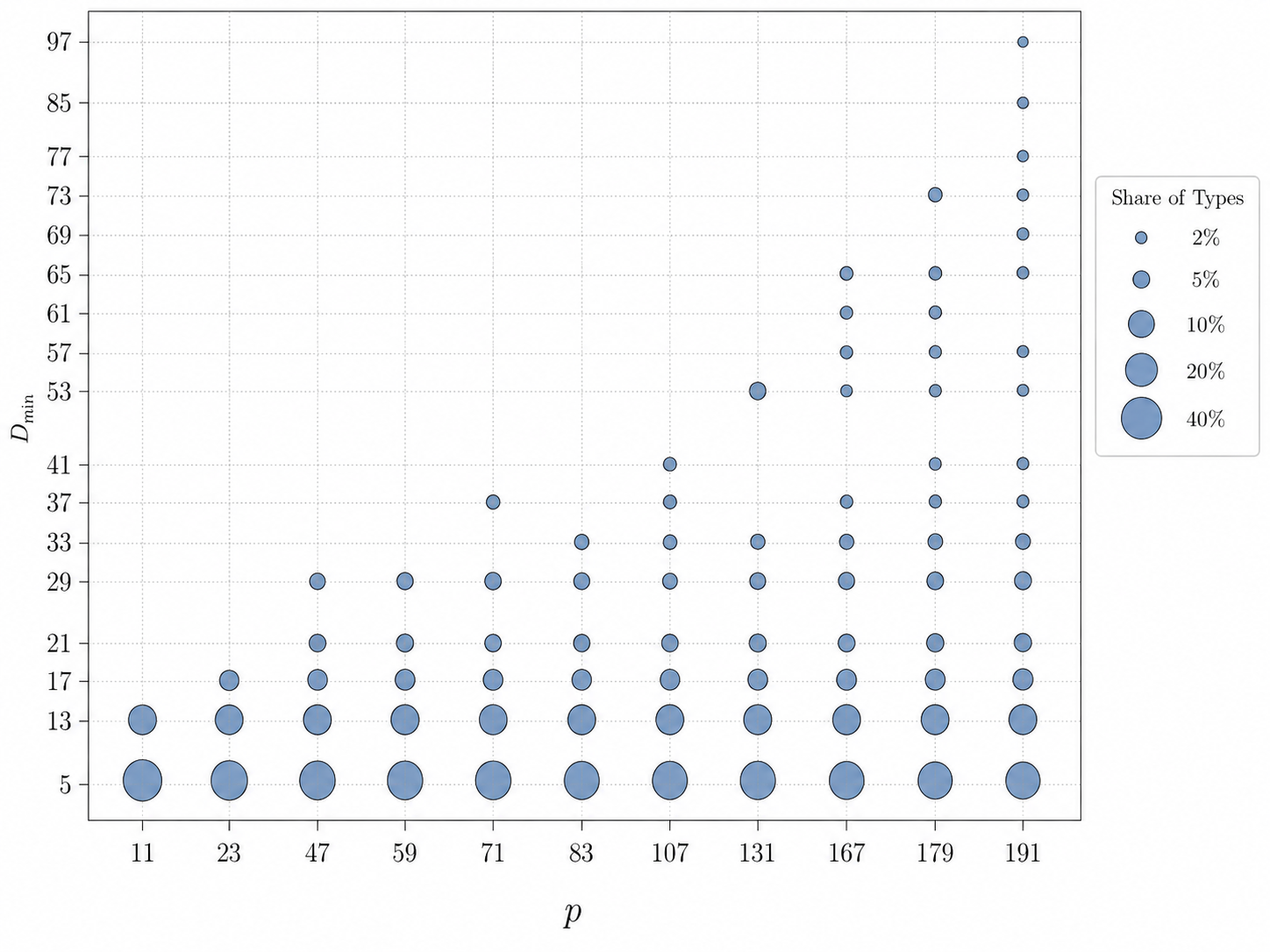}
        \smallskip
        \textbf{(a)}
    \end{minipage}
    \hfill
    \begin{minipage}[t]{0.49\textwidth}
        \centering
        \includegraphics[width=\linewidth]{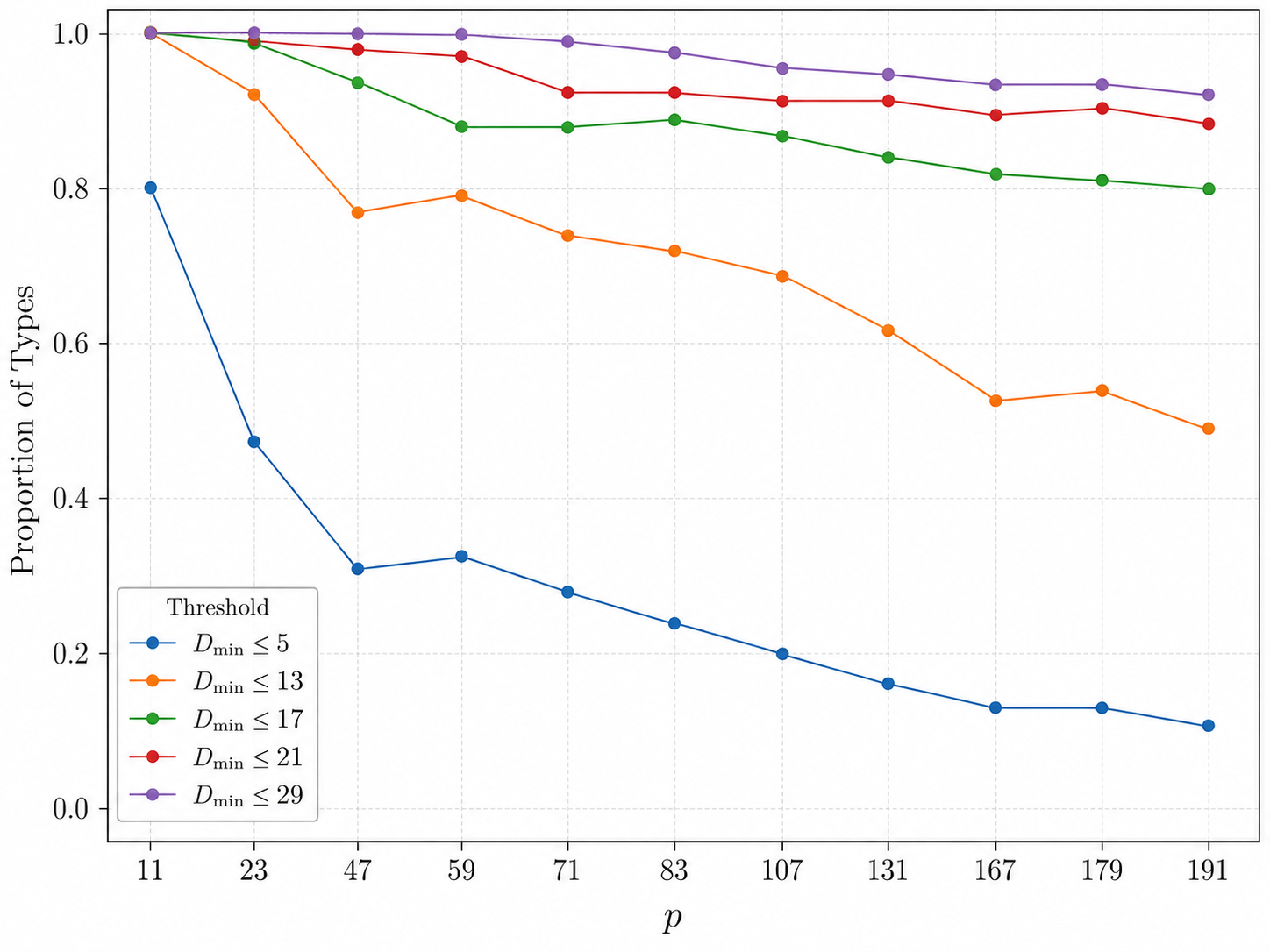}

        \smallskip
        \textbf{(b)}
    \end{minipage}
    \caption{%
    \textbf{(a)} Distribution of the minimum primitively represented value
    $D_{\min}$ for each prime $p$. For each Type, the value $D=1$ is excluded
    and $D_{\min}$ denotes the least remaining primitively represented value.
    For each pair $(p,D_{\min})$, the marker area is proportional to the
    fraction of Types attaining that minimum.
    \\
    \textbf{(b)} Cumulative proportions of Types satisfying prescribed bounds
    on $D_{\min}$. For each prime $p$, after excluding $D=1$, the curves give
    the proportion of Types for which $D_{\min}\leq D_0$, with
    $D_0\in\{5,13,17,21,29\}$.}
    \label{fig:min-D-summary}
\end{figure}

\subsection{The experiment with random polarizations for small primes}
For each prime $p\equiv11\pmod{12}$ with $227\leq p\leq1619$, we sample $100{,}000$ random
polarizations with Algorithm~\ref{alg:polzrandom} and compute the corresponding refined Humbert
invariants. For every resulting Type, we test
primitive representations of the square-free integers $D\leq 100$. Since
$D=1$ corresponds to the trivial case, we exclude it and define $D_{\min}(T):=\min\{D>1 : D \text{ is primitively represented by }q_T\}$.

Figure~\ref{fig:min-D-summary-random}(a) describes the distribution of the
first nontrivial represented value. For each prime $p$ and value $D$, the
marker area is proportional to the fraction of recorded Types satisfying
$D_{\min}=D$. The experiment, therefore, measures how early, with respect to
the discriminant bound, a nontrivial primitive representation is encountered.
The distribution is strongly concentrated at small values, principally
$D_{\min}=5,13,$ and $17$, whereas Types whose first primitive representation
occurs at a substantially larger discriminant form only a sparse upper tail.

Figure~\ref{fig:min-D-summary-random}(b) gives a cumulative description of
the same data. For each threshold $D_0\in\{5,13,17,21,29\}$, we compute
$F_p(D_0):=
\frac{\#\{T : D_{\min}(T)\leq D_0\}}
     {\#\{T\}}$, where $T$ ranges over the recorded Types obtained from the random-polarization
experiment for the prime $p$. Thus, $F_p(D_0)$ is the empirical proportion
for which a nontrivial primitive representation has already occurred by the
time the search reaches $D_0$. The curves show that the proportion attaining
the very smallest values decreases as $p$ increases. Nevertheless, a large
fraction of the Types still satisfy $D_{\min}\leq 17$, and allowing the
threshold to increase to $21$ or $29$ captures an even larger proportion.
Hence, throughout the tested range of primes, the first nontrivial primitive
representation is typically found at a comparatively small discriminant.

This experiment also serves as a partial check on the sampler. Figure~\ref{fig:min-D-summary}
and Figure~\ref{fig:min-D-summary-random} record the same statistic, the distribution of
$D_{\min}$ over Types, computed once from the complete set of principal polarizations
($11\leq p\leq191$) and once from $100{,}000$ draws of Algorithm~\ref{alg:polzrandom}
($227\leq p\leq503$). The two ranges do not overlap, so we cannot compare the two
directly; what we can say is that the sampled curves continue the exhaustive ones without
a visible discontinuity at the join, with $F_p(5)$ passing from $0.11$ at $p=191$ to
$0.10$ at $p=227$, and that $D_{\min}$ concentrates on $5$, $13$ and $17$ in both. A
like-for-like comparison would require running the sampled experiment at a prime where
the exhaustive list of \cite{edagaurish} is also available; we have not done so, and the
distribution Algorithm~\ref{alg:polzrandom} induces on isometry classes therefore remains
uncharacterized, as noted in Section~\ref{future}.

\begin{figure}[hbt]
    \centering

    \begin{minipage}[t]{0.49\textwidth}
        \centering
        \includegraphics[width=\linewidth]{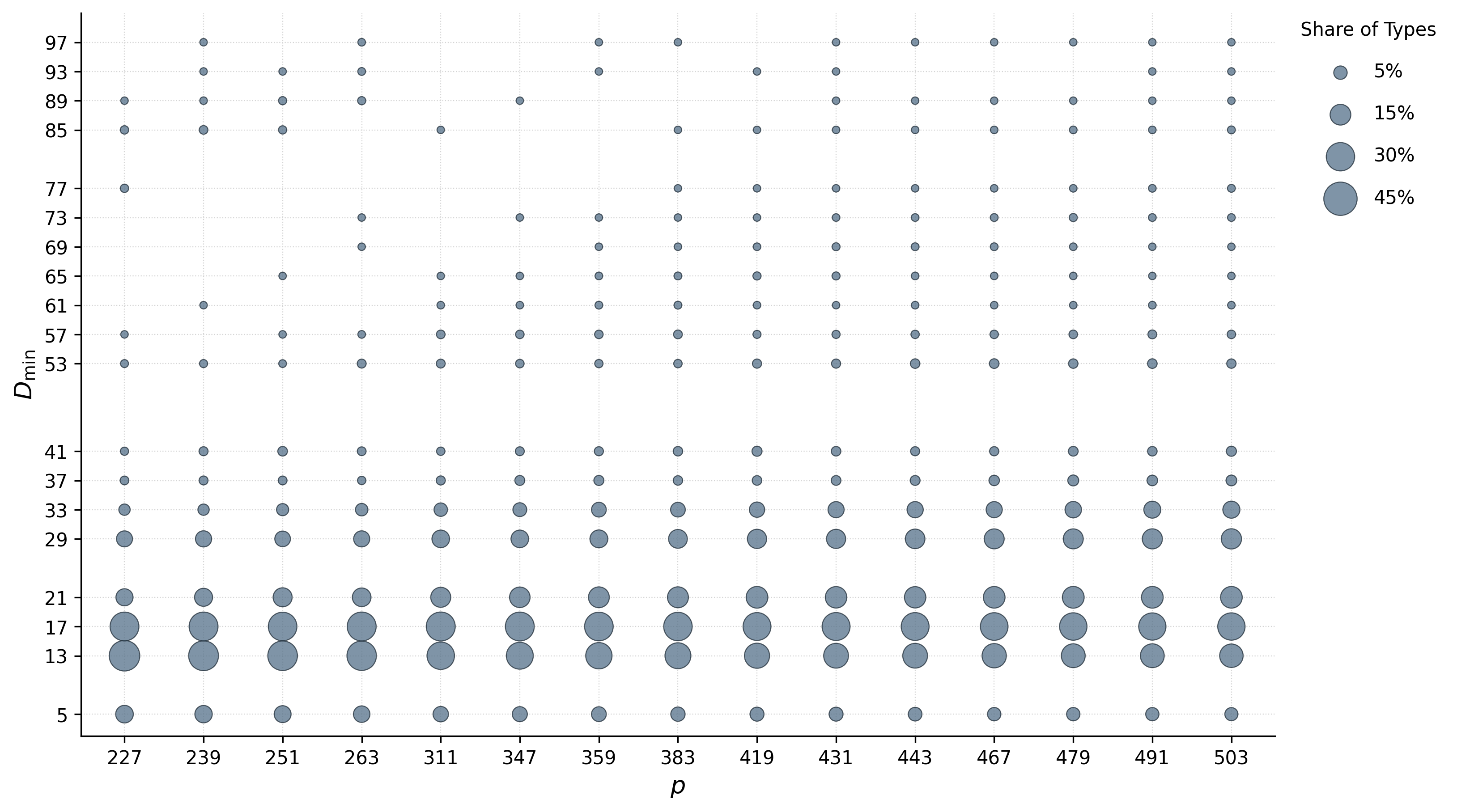}

        \smallskip
        \textbf{(a)}
    \end{minipage}
    \hfill
    \begin{minipage}[t]{0.49\textwidth}
        \centering
        \includegraphics[width=\linewidth]{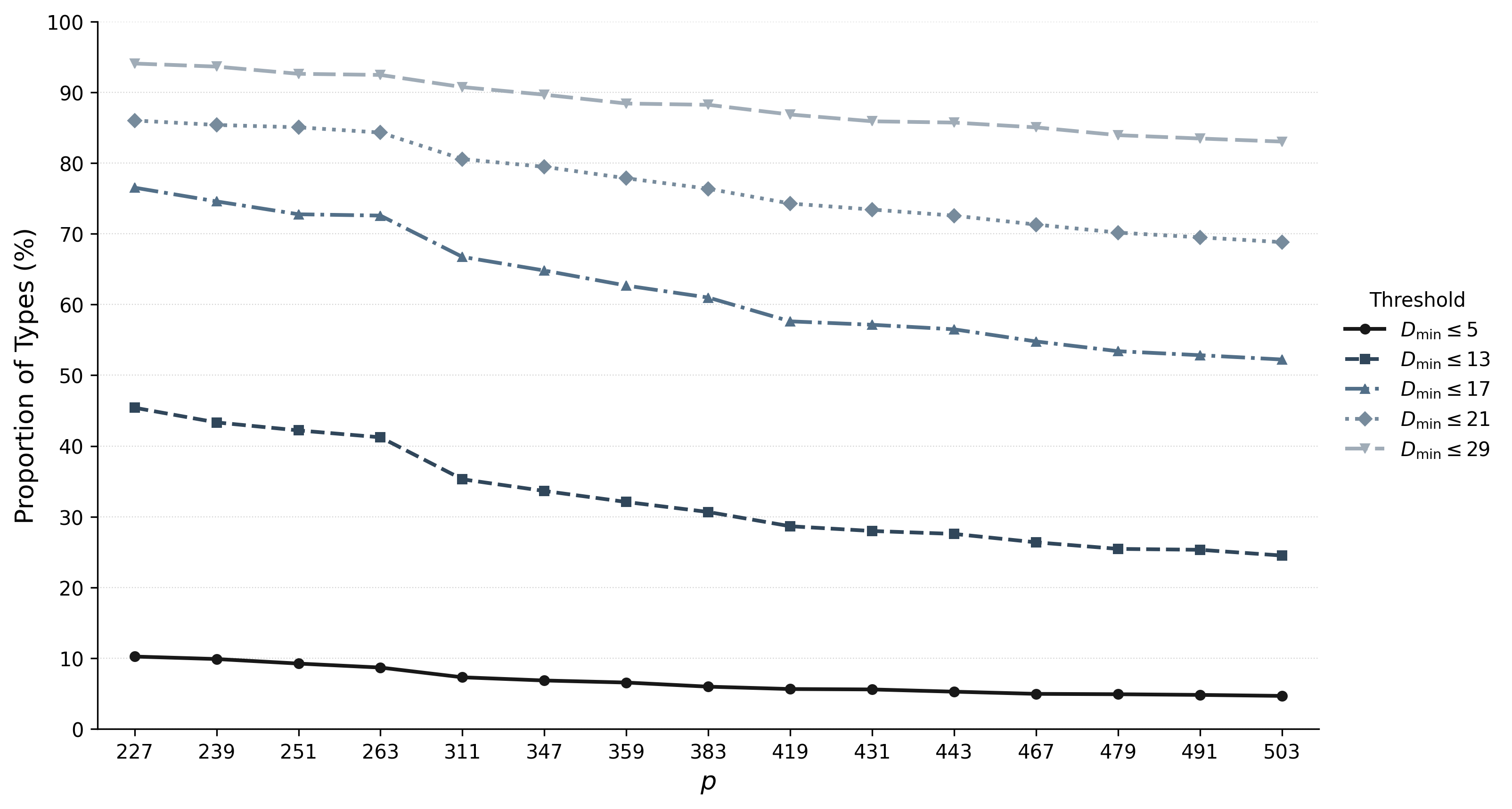}

        \smallskip
        \textbf{(b)}
    \end{minipage}

    \caption{%
    \textbf{(a)} Empirical distribution of the minimum primitively represented value $D_{\min}$ obtained from the $100{,}000$-polarization sampling experiment. For each recorded Type, we exclude $D=1$ and define $D_{\min}$ to be the least remaining primitively represented square-free value. For each pair $(p,D_{\min})$, the marker area is proportional to the fraction of recorded Types attaining that minimum. Types for which no value $D>1$ is represented up to $100$ are omitted from this distribution. \\
    \textbf{(b)} Cumulative proportions of the recorded Types satisfying prescribed bounds on $D_{\min}$. For each prime $p$ and threshold $D_0\in\{5,13,17,21,29\}$, the plotted value is the proportion of Types obtained from the $100{,}000$-polarization sampling experiment for which $D_{\min}\leq D_0$. Thus, each curve measures how frequently a first nontrivial primitive representation occurs by the corresponding threshold $D_0$.}
    \label{fig:min-D-summary-random}
\end{figure}

\paragraph{Experimental distribution of RM vertices.}
We summarize the experimental data in \cite{detectionofendomorphisms} for the distribution of RM surfaces by $\mathcal{O}_D$ in the isogeny graph of principally polarized superspecial abelian surfaces. The observed values suggest that the leading term in \eqref{RMheuristic} can be more accurately approximated by $p/\widetilde{\mathfrak{D}}_D^{\mathrm{expt}}$. These experiments appear to suggest that the leading
term in \eqref{RMheuristic} should be closer to $p/\widetilde{\mathfrak{D}}_D^{\mathrm{expt}}$ where
\begin{equation}
\widetilde{\mathfrak{D}}_D^{\mathrm{expt}}
=
\begin{cases}
\tfrac{5}{4}\mathfrak{D}_N
    & \text{if } D=N^2, \text{ and}\\
\tfrac{1}{3}\cdot
\tfrac{360\zeta_K(2)}{\pi^4}D^{3/2}
    & \text{if } D \text{ is fundamental.}
\end{cases}
\end{equation}

The condition of narrow class number being 1 is not needed for us. Nevertheless, the heuristic still aligns with our experimental data, 
even when $\mathcal{O}_D$ does not have narrow class number 1.

\subsection{Large primes}

The two experiments reported here are independent. The first fixes a single $50$-bit prime and
measures the least splitting level $k_0$ with Algorithm~\ref{alg:kzero}; the second ranges over
twenty primes of up to $1000$ bits and measures the splitting degrees returned by
Algorithm~\ref{alg:witness}. Their samples were drawn separately.

\paragraph{The least level $k_0$ at a $50$-bit prime.}
Fix $p=2^{11}\cdot3^{24}-1$, of $50$ bits, and $\ell=2$. We sample $10{,}000$ classes with
irreducible polarization from the genus of $L_5=(\Z^{5},q_5)$ --- refined Humbert invariants of
determinant $16p^{2}$ --- and compute $k_0$ for each.

The lattice minima are close to extremal. Here $(\det G_q)^{1/5}=1.3987\cdot10^{6}$, so Hermite's
bound $\gamma_5(\det G_q)^{1/5}$ with $\gamma_5=8^{1/5}=1.5157$ is $2.1200\cdot10^{6}$, while the
largest minimum in the sample is $\lambda_{1}=1.214\,(\det G_q)^{1/5}$. Every class therefore has a
minimum within $20\%$ of the smallest attainable by a rank-five lattice of this determinant, so the
barrier $k_1$ of \eqref{eq:k1bound} is set by the forms themselves and not by slack in the bound.

\begin{figure}[ht]
\centering
\includegraphics[width=\textwidth]{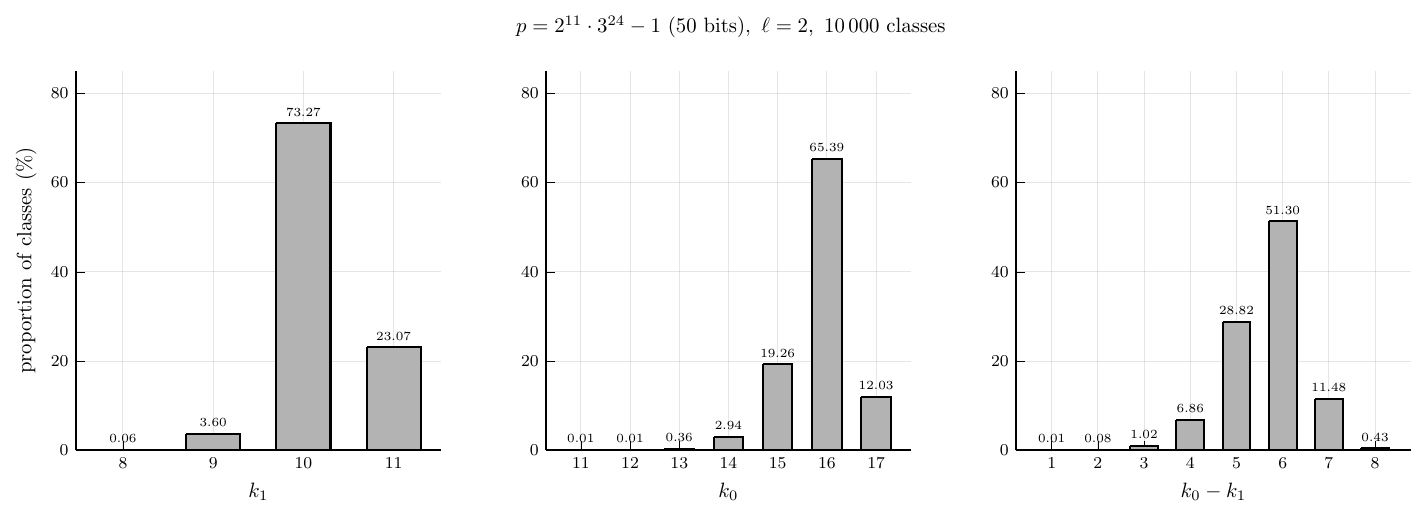}
\caption{The $10{,}000$ sampled classes at $p=2^{11}\cdot3^{24}-1$ with $\ell=2$: the barrier $k_{1}$
of \eqref{eq:k1bound}, the least splitting level $k_{0}$, and their difference. The barrier
concentrates on the leading term $\tfrac15\log_2 p$; the level lies some six steps above it; the two
never coincide.}
\label{fig:compare}
\end{figure}

Figure~\ref{fig:compare} gives the three distributions. The barrier is sharply concentrated:
$73.27\%$ of the classes have $k_1=10$ and $23.07\%$ have $k_1=11$, the remaining $3.66\%$ lying at
$8$ or $9$. The leading term of \eqref{eq:k1bound} is $\tfrac15\log_2 p=9.81$ at this prime, so
almost every class attains that estimate to within one step. The level $k_0$ sits some six steps
higher: $65.39\%$ of the classes have $k_0=16$ and a further $31.29\%$ have $k_0=15$ or $17$, with a
thin tail reaching down to $k_0=11$. The mean is $15.86$, that is $0.32\log_2 p$.

The third panel shows that the two never coincide. The smallest difference in the sample is
$k_0-k_1=1$, attained by one class in $10{,}000$; the distribution is centred on $6$ ($51.30\%$) and
$5$ ($28.82\%$), with mean $5.66$. The reason is that the targets $4^{k}$ are sparse. The barrier
certifies only that $4^{k_1}$ is not excluded by the lattice minimum, whereas a representation
requires $r(L,4^{k})>0$. For a rank-five lattice $r(L,m)\asymp m^{3/2}/\sqrt{\det G_q}$, so the
shells of norm $4^{k}$ stay empty until $4^{3k/2}\gtrsim\sqrt{\det G_q}$, that is until
$k\gtrsim\tfrac16\log_2\det G_q$; requiring the representing vector to be primitive removes a
further thin set. Against $k_1\approx\tfrac1{10}\log_2\det G_q$ this predicts a difference growing
like $\tfrac1{15}\log_2\det G_q$, which is the order of magnitude observed.

The same count gives $k_0\approx\tfrac16\log_2(16p^{2})$, so $k_0$ should grow linearly in
$\log_2 p$ with slope near $\tfrac13$; the measured $0.32\log_2 p$ agrees with this. The observed
value falls far short of the diameter bounds of
Section~\ref{sec:isogenygraphs}, which at this prime give $3\log_2 p=147$ conditionally and
$4.20\log_2 p=206$ unconditionally, so the stopping criterion of Section~\ref{sec:primreps} never
binds in practice.

\paragraph{Unrestricted splitting degrees at cryptographic sizes.}
Next, we apply Algorithm~\ref{alg:witness} to the primes in \cite[Table 6]{detectionofendomorphisms},
drawing $100{,}000$ classes independently for each prime.

\begin{figure}[htbp]
\centering
\includegraphics[width=0.49\textwidth]{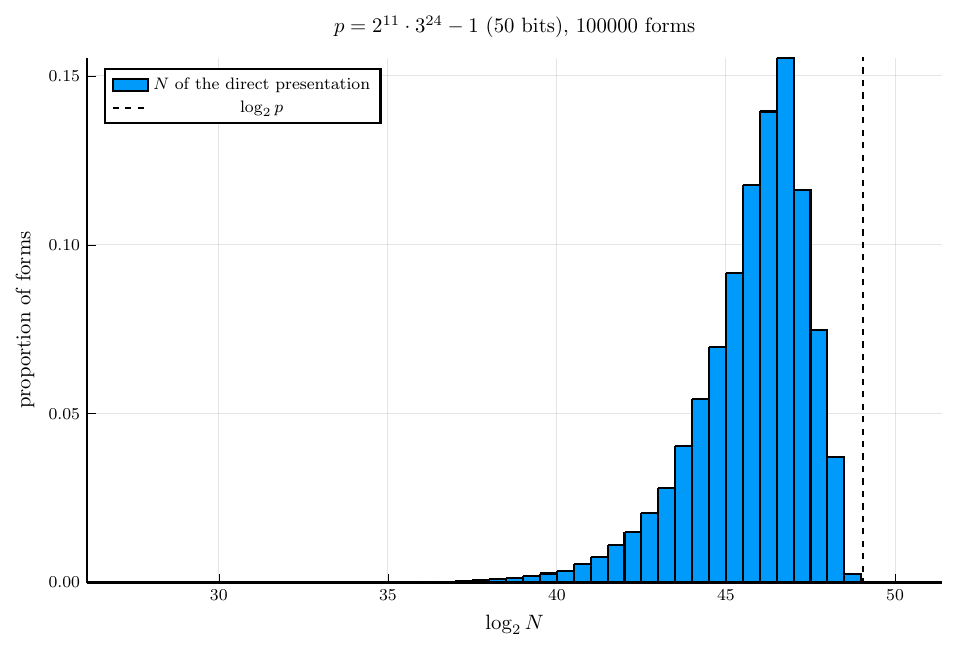}
\includegraphics[width=0.49\textwidth]{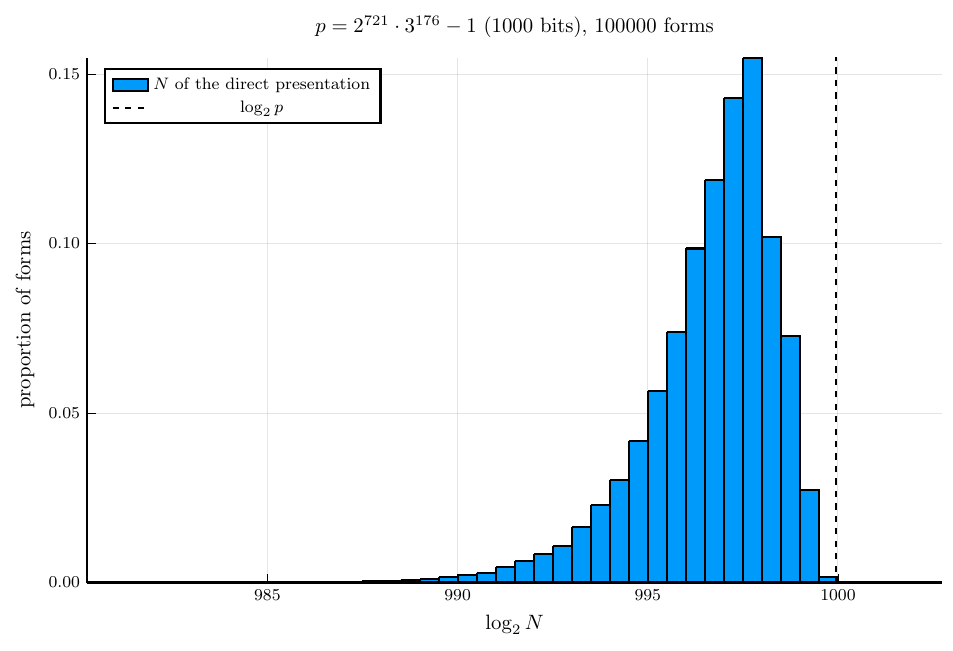}
\caption{The histogram of $\log_{2}N$ at $50$ bits (left) and $1000$ bits
(right), from the per-prime reports; the dashed line is $\log_{2}p$. The
same shape, shifted by $950$ bits.}\label{fig:sizes2}
\end{figure}

Figure~\ref{fig:sizes2} carries the finding of most direct cryptographic relevance. In both panels the
histogram of $\log_2 N$ concentrates on the dashed line $\log_2 p$, and the shape of the
distribution is essentially unchanged between $50$ and $1000$ bits. The witness returned by
Algorithm~\ref{alg:witness} therefore has $N$ of roughly the size of $p$, so the $(N,N)$-splitting
it detects has degree about $p^{2}$, uniformly across the whole range of parameters we tested.

Write $P^*(N)=\max\{\ell^{e}:\ell^{e}\parallel N\}$ for the largest prime-power
divisor of $N\ge2$, with $P^*(1)=1$. For $B\ge2$, $N$ is \emph{$B$-power-smooth} if
$P^*(N)\le B$; this is the notion relevant to isogeny computation and the one
tabulated below.

\begin{table}[htbp]
\centering
\tiny
\setlength{\tabcolsep}{2pt}
\begin{tabular}{cccccp{3.3in}}
\toprule
bits & $p$ & class & $\log_{2}N$ & $\log_{2}P^*(N)$ & $N$\\
\midrule
50 & $2^{11}\cdot 3^{24}-1$ & 89\,352 & 42.1 & 7.5 & $2^{3}\cdot 3^{3}\cdot 5\cdot 23\cdot 73\cdot 103\cdot 137\cdot 181$ \\
100 & $2^{44}\cdot 3^{35}-1$ & 71\,177 & 94.8 & 15.5 & $2\cdot 131\cdot 241\cdot 269\cdot 8219\cdot 13177\cdot 13229\cdot 30113\cdot 45893$ \\
150 & $2^{27}\cdot 3^{77}-1$ & 87\,317 & 147.3 & 24.5 & $2\cdot 223\cdot 1511\cdot 3323\cdot 48029\cdot 163433\cdot 165041\cdot 489761\cdot 6480421\cdot 24222791$ \\
200 & $2^{144}\cdot 3^{35}-1$ & 41\,518 & 195.3 & 28.4 & $2\cdot 23431\cdot 363833\cdot 505927\cdot 1131103\cdot 1143013\cdot 3094361\cdot 68911559\cdot 69813571\cdot 363789383$ \\
250 & $2^{181}\cdot 3^{43}-1$ & 59\,632 & 247.4 & 35.5 & $2^{6}\cdot 9293\cdot 482753\cdot 5258807\cdot 23598691\cdot 255769463\cdot 849190603\cdot 28251002411\cdot 29155313437\cdot 48311528479$ \\
300 & $5\cdot 2^{193}\cdot 3^{66}-1$ & 1\,765 & 296.0 & 184.0 & $2\cdot 3^{5}\cdot 11\cdot 17\cdot 179\cdot 38821\cdot 84011\cdot 300491\cdot 388253\cdot 828967\cdot \nu$, $\nu$ prime of 184 bits \\
350 & $2^{201}\cdot 3^{94}-1$ & 50\,392 & 345.2 & 252.0 & $23\cdot 2153\cdot 2677\cdot 2789\cdot 265613\cdot 294461\cdot 398287\cdot \nu$, $\nu$ prime of 252 bits \\
400 & $2^{231}\cdot 3^{106}-1$ & 84\,890 & 396.7 & 308.8 & $2\cdot 5\cdot 263\cdot 283\cdot 541\cdot 1423\cdot 13691\cdot 41729\cdot 891893\cdot \nu$, $\nu$ prime of 309 bits \\
450 & $2^{204}\cdot 3^{155}-1$ & 68\,535 & 446.6 & 353.6 & $3\cdot 11\cdot 17\cdot 9137\cdot 82267\cdot 93323\cdot 357169\cdot 752351\cdot \nu$, $\nu$ prime of 354 bits \\
500 & $2^{113}\cdot 3^{244}-1$ & 96\,064 & 495.2 & 390.3 & $5^{4}\cdot 7\cdot 37\cdot 157\cdot 15373\cdot 34403\cdot 114493\cdot 134293\cdot 182857\cdot \nu$, $\nu$ prime of 391 bits \\
550 & $2^{293}\cdot 3^{162}-1$ & 31\,012 & 548.9 & 446.0 & $2^{2}\cdot 17\cdot 6203\cdot 20939\cdot 50231\cdot 63709\cdot 502087\cdot 622861\cdot \nu$, $\nu$ prime of 447 bits \\
600 & $5\cdot 2^{299}\cdot 3^{188}-1$ & 22\,865 & 590.7 & 490.3 & $3\cdot 71\cdot 83\cdot 107\cdot 139\cdot 1327\cdot 8663\cdot 28349\cdot 36137\cdot 528779\cdot \nu$, $\nu$ prime of 491 bits \\
650 & $2^{404}\cdot 3^{155}-1$ & 51\,632 & 645.8 & 526.2 & $2^{6}\cdot 3\cdot 31\cdot 103\cdot 107\cdot 197\cdot 223\cdot 3529\cdot 30091\cdot 47713\cdot 95731\cdot 701419\cdot \nu$, $\nu$ prime of 527 bits \\
700 & $2^{83}\cdot 3^{389}-1$ & 7\,620 & 694.5 & 598.8 & $11\cdot 151\cdot 379\cdot 1657\cdot 12043\cdot 28351\cdot 279607\cdot 667559\cdot \nu$, $\nu$ prime of 599 bits \\
750 & $2^{477}\cdot 3^{172}-1$ & 61\,730 & 745.6 & 659.2 & $2^{2}\cdot 1861\cdot 4057\cdot 7937\cdot 41467\cdot 50363\cdot 200689\cdot \nu$, $\nu$ prime of 660 bits \\
800 & $2^{107}\cdot 3^{437}-1$ & 36\,022 & 795.1 & 707.3 & $5\cdot 53\cdot 73\cdot 79\cdot 137\cdot 1657\cdot 44017\cdot 83071\cdot 207293\cdot \nu$, $\nu$ prime of 708 bits \\
850 & $2^{166}\cdot 3^{431}-1$ & 49\,054 & 845.4 & 751.0 & $2\cdot 5^{7}\cdot 13\cdot 59^{2}\cdot 5279\cdot 42677\cdot 50527\cdot 337217\cdot \nu$, $\nu$ prime of 751 bits \\
900 & $2^{172}\cdot 3^{459}-1$ & 44\,512 & 895.2 & 801.5 & $2^{5}\cdot 227\cdot 823\cdot 1931\cdot 4127\cdot 18341\cdot 62861\cdot 277897\cdot \nu$, $\nu$ prime of 802 bits \\
950 & $2^{536}\cdot 3^{261}-1$ & 35\,055 & 943.3 & 854.3 & $137\cdot 1433\cdot 93337\cdot 163411\cdot 306407\cdot 662713\cdot \nu$, $\nu$ prime of 855 bits \\
1000 & $2^{721}\cdot 3^{176}-1$ & 57\,274 & 997.2 & 896.7 & $11\cdot 197\cdot 2797\cdot 7499\cdot 19051\cdot 24337\cdot 102983\cdot 822761\cdot \nu$, $\nu$ prime of 897 bits \\
\bottomrule
\end{tabular}
\caption{The witness of least $P^{*}(N)$ found at each prime, over the $100\,000$ classes sampled
for that prime. The \emph{class} column is the index of the class realizing it.}\label{tab:smoothest}
\end{table}

Table~\ref{tab:smoothest} records a negative result as much as a positive one. Up to $250$ bits the
smoothest witness found is genuinely power-smooth, with $\log_2P^{*}(N)$ between $7.5$ and $35.5$.
From $300$ bits upward, however, the best of the $100\,000$ witnesses sampled at that prime still carries a prime factor of
between $184$ and $897$ bits, and $\log_2P^{*}(N)$ tracks $\log_2N$ closely. In other words,
Algorithm~\ref{alg:witness} as it stands essentially never returns a power-smooth splitting degree at
cryptographic size. Constraining Simon's algorithm so that it does is the first of the open
problems in Section~\ref{future}.

\section{Conclusion and future directions}\label{future}

We have shown that two structural questions about principally polarized superspecial abelian
surfaces --- whether a surface admits an $(N,N)$-splitting, and whether it carries real
multiplication by a prescribed order --- reduce to primitive-representation problems for the
quintic refined Humbert invariant, and that both can be answered from the polarization data alone,
without Igusa invariants, defining equations for a genus $2$ curve, or any explicit chain of
isogenies. Section~\ref{designing} turns this into three algorithms: a sampler for principal polarizations
that reaches far beyond exhaustive enumeration, an enumeration procedure for the least splitting
level $k_0$, and a reduction of the unrestricted splitting problem to finding an isotropic vector
of an indefinite rank-six form. Lemma~\ref{lem:k0-distance} places $k_0$ precisely: it is the length of
the shortest good-extension walk from the surface to the split surface, and an upper bound for the
distance in the Richelot graph.

The experiments support this in two regimes. For $11\leq p\leq251$, where every principal
polarization is known, the automorphism counts obtained from $r_4(q_{\C})$ reproduce the counts of
Ibukiyama--Katsura--Oort exactly at the level of column totals (Table~\ref{auto-experiment} and
Appendix~\ref{app:iko}), and the irreducible polarizations of $E\times E$ turn out to realize \emph{every}
class in the genus of $q_5$ that does not represent $1$.
For large parameters, Algorithm~\ref{alg:witness} produced witnesses for primes of up to $1000$
bits; the resulting splitting degrees satisfy $\log_2 N\approx\log_2 p$, with a distribution whose
shape is essentially unchanged across the range, so the detected $(N,N)$-isogeny has degree about
$p^{2}$ throughout. Table~\ref{tab:smoothest} records the corresponding negative result: power-smooth
witnesses occur up to $250$ bits and, among the $100\,000$ classes sampled per prime, not beyond $300$.

Two limitations should be stated plainly. First, our method detects a splitting without
constructing it: a primitive representation certifies that an $(N,N)$-isogeny to a product exists
and fixes its degree, but yields neither the isogeny nor the identity of the surfaces involved.
Second, the experiments in the sampled regime inherit whatever distribution
Algorithm~\ref{alg:polzrandom} induces on isomorphism classes, which we have not characterized;
conclusions drawn from those samples should be read with that in mind.

One direction is to impose arithmetic conditions on the splitting degree $N$ arising from primitive representations $q_{(\A,\theta)}(x)=N^2$. For isogeny computation, the most relevant cases are $N=\ell^k$ and, more generally, $B$-power-smooth $N$. If $N=\prod_i\ell_i^{e_i}$, then $N$ is $B$-power-smooth when $\ell_i^{e_i}\leq B$ for every $i$; this is stronger than $B$-smoothness, which only requires $\ell_i\leq B$. Our computations already produce primitive representations with $B$-power-smooth $N$ at moderate size: among the $100{,}000$ classes sampled per prime, the smoothest witness has $\log_2P^{*}(N)$ between $7.5$ and $35.5$ for primes up to $250$ bits, but from $300$ bits upward none below $184$ bits was found (Table~\ref{tab:smoothest}). A natural problem is therefore to modify the randomized choices, reductions, and restarts in Simon's algorithm~\cite{simon,watkins} so as to produce primitive solutions $q_{(\A,\theta)}(v)=N^2$ with $N$ constrained to be an $\ell$-power or $B$-power-smooth. This amounts to obtaining effective bounds for the smallest such $N$, together with an efficient algorithm for finding the corresponding primitive representation.

A second direction is to investigate whether the splittings via refined Humbert invariant can be used to improve and accelerate $\mathrm{KLPT}^2$, in the same spirit as the application of conversion of module representations by Page-Robert--Soumier \cite{AurelJulienDamien}. Their approach passes from the Hermitian module representation to the underlying rank-$8$ integral quadratic lattice, where suitable primitive vectors are used to construct principal rank-$1$ submodules and corresponding splitting isogenies; as an application, this leads to an improved output-degree bound for $\mathrm{KLPT}^2$.
Starting from the same IKO representation, the principal polarization determines the quintic refined Humbert invariant $q_{(\A,\theta)}$. A primitive representation $q_{(\A,\theta)}(x)=N^2$ determines a splitting degree $N$, so the search for suitable values of $N$ can instead be formulated directly as a primitive representation problem of $N^2$ by $q_{(\A,\theta)}$. This separates the determination of the splitting degree from the explicit construction of the corresponding isogeny. It remains to determine whether the $(N,N)$-splitting detected by a primitive representation of $q_{\C}$ can also be constructed effectively at the level of explicit isogenies, and whether this approach can be used to obtain a comparable improvement of $\mathrm{KLPT}^2$.

\noindent\textbf{Acknowledgements:}
The authors used Claude (Anthropic) and ChatGPT (OpenAI) exclusively for text proofreading and minor code debugging. All mathematical formulations, proofs, and algorithmic contributions were developed independently by the authors.

\bibliographystyle{splncs04}
\bibliography{bibliography.bib}

\appendix
\section{Refined Humbert invariants}\label{sec:RHI}
In this section, we introduce the refined Humbert invariants defined by Kani \cite{Kani94,KaniMJ}. The refined Humbert invariant $q_{(\A,\theta)} $ is the main tool of this paper. These invariants are very useful for the interplay between geometric and arithmetic problems. This section mainly follows Kani \cite{KaniJac14,KaniMJ}.

We first present some essential facts about the N\'eron-Severi groups of abelian varieties by following \cite[Section 11]{KaniMJ}. 

\begin{definition}\label{defn: NS}
Let $\A$ be an abelian variety over a field $K$. The \emph{N\'eron-Severi group} of $\A$ is defined as the quotient group
$\NS(\A) = \Pic(\A)/ \Pic^{0}(\A)$, which is a free $\Z$-module of finite rank.
\end{definition}

\begin{definition}\emph{\cite[p.357]{hartshorne2013algebraic}}\label{divisorequivalence2}
Let $\A$ be an abelian surface over a field $K$ and let $\Div(\A)$ be the set of divisors on $\A$.
For two divisors $D_1$, $D_2 \in \Div(\A)$, we say that $D_1$ is \emph{numerically equivalent} to $D_2$, denoted by $D_1 \equiv D_2$, if
$(D_1 \cdot D) = (D_2 \cdot D)$ for all $D \in \Div(\A)$,
where $(\cdot)$ denotes the intersection number.
\end{definition}

\begin{definition}\label{polarizations}
Let $\A/K$ be an abelian surface. We define \emph{the set of principal polarizations} of $\A$ as
\begin{align*}
\mathcal{P}(\A)= \{\cl(D)\in \NS(\A) : D\in \Div(\A) \text{ is ample and } (D\cdot D) =2 \}.
\end{align*}
\end{definition}

For a divisor $D\in\Div(\A)$, we write $\cl(D)$ for its class in $\NS(\A)$ and $\phi_D:\A\rightarrow\hat{\A}$ for the homomorphism associated to $D$. When $\cl(D)\in\mathcal{P}(\A)$, the map $\phi_D$ is an isomorphism. Throughout, we use the same symbol $\theta$ for a principal polarization class in $\NS(\A)$ and for the corresponding isomorphism $\theta:\A\xrightarrow{\sim}\hat{\A}$ when no confusion can arise.

If $\A$ has a principal polarization $\theta : \A  \xrightarrow{\sim},
\hat{\A}$  then $\NS(\A)$ can be interpreted as a subgroup
of $\End(\A)$  \cite[p.~126]{milne1986abelian}. Moreover, if $r_\theta$ denotes the Rosati involution on $\End(\A)$,  defined by $r_{\theta}(\mu) = \theta^{-1}
\widehat{\mu} \theta$, then by \cite[p.~189,190]{mumford1970abelian}, the map
$D \mapsto \theta^{-1}\phi_D$ defines an isomorphism
\begin{equation}\label{morphismEndNS}
\Phi_{\theta} : \NS(\A)  \xrightarrow{\sim} \End_{\theta}(\A) := \{ \mu \in \End(\A) : r_{\theta}(\mu) = \mu \}.
\end{equation}
The intersection number $(D_1\cdot D_2)$ of divisors $D_1$, $D_2$ on $\A$ defines an integral quadratic form as follows.
\begin{definition}\label{intersectionform}
The intersection product $(D_1 \cdot  D_2)$ of divisors $D_1, D_2$ on an abelian surface $\A$ defines an integral quadratic form $q_{\A}$ on $\NS(\A)$, called \emph{intersection form}: $q_{\A}(D)=\frac{1}{2}(D\cdot D)$ for all $D\in \NS(\A)$.
\end{definition}
Since $\NS(\A)\cong \Z^{\rho}$ where $\rho = \rho(\A)$ is the Picard number of $\A$, $q_{\A}$ is equivalent to an integral quadratic form $q$ in $\rho$ variables, therefore it leads to an isomorphism $(\NS(\A), q_{\A})\cong (\Z^{\rho}, q)$ of quadratic modules.

\begin{definition}\label{polNS}
Let $\A/K$ be an abelian surface and let $\theta \in \mathcal{P}(\A)$ be a principal polarization of $\A$. We define the \emph{polarized N\'eron-Severi group} of $(\A,\theta)$ to be    $\NS(\A,\theta) \coloneqq \NS(\A)/\Z\theta$.
\end{definition}
If $\A$ has a principal polarization $\theta\in\mathcal{P}(\A)$, then the quadratic form $\tilde{q}_{(\A,\theta)}$ on $\NS(\A)$ is defined by
\begin{equation}
\tilde{q}_{(\A,\theta)}(D) = (D \cdot \theta)^2 - 2(D\cdot D), \textrm{ for } D \in \NS(\A).
\end{equation}
It follows that $\tilde{q}_{(\A,\theta)}(D + n\theta) = \tilde{q}_{(\A,\theta)}(D)$ for all $n\in \Z$, therefore $\tilde{q}_{(\A,\theta)}$ actually gives a quadratic form $q_{(\A,\theta)}$ on the quotient module $\NS(\A, \theta)$. Then, $q_{(\A,\theta)}$ is a positive definite form on $\NS(\A,\theta)\cong \Z^{\rho-1}$ \cite{Kani94}.

This is the refined Humbert invariant $q_{(\A,\theta)}$ of Definition~\ref{RHIdefn}.

For an integral quadratic form $q$ on a lattice $L$, we say that $q$ \emph{represents} an integer $n$ if $q(x)=n$ for some $x\in L$, and that $q$ \emph{primitively represents} $n$ if such a vector $x$ can be chosen primitive. We write $q\equiv 0,1\pmod 4$ to mean that $q(x)\equiv 0$ or $1\pmod 4$ for every $x\in L$.

\paragraph{A necessary condition} For a generic principally polarized abelian surface $(\A,\theta)$, any integral quadratic form arising as a refined Humbert invariant necessarily satisfies the following condition.

\begin{theorem} \label{necessarycondition}
If there exists $(\A,\theta)/K$ such that a given integral quadratic form $q$ is equivalent to a refined Humbert invariant $q_{(\A,\theta)}$, then $q \equiv 0,1\pmod 4$.
\end{theorem}
\begin{proof}
See~\cite[Theorem 1]{Kanirefhum}. This result appeared in \cite[Theorem 3.1.1]{harunkirthesis}
\end{proof}

\paragraph{An equivalent definition:}  Kani generalized and extended the definition of a refined Humbert invariant in \cite[Remark 16]{KaniESCI} later.
By the isomorphism $\Phi_{\theta}$ \ref{morphismEndNS} above, the polarization class $\theta$ corresponds to the identity endomorphism $1_{\A}$. Hence $\Phi_{\theta}$ induces an isomorphism $
\NS(\A,\theta)=\NS(\A)/\Z\theta\xrightarrow{\sim}\End_{\theta}(\A)/\Z 1_{\A}$. Through this identification, we use the same notation $q_{(\A,\theta)}$ for the corresponding quadratic form on $\End_{\theta}(\A)/\Z1_{\A}$, and write $q_{(\A,\theta)}(\mu)$ for the value on the class of $\mu$.

\begin{proposition}\label{RHInewdefn}
Let $(\A,\theta)$ be a principally polarized abelian surface over a field $K$ and let $q_{(\A,\theta)}$ be a refined Humbert invariant of $(\A,\theta)$.
Then for every $\mu \in \End_{\theta}(\A)$, we obtain
$$q_{(\A,\theta)}(\mu) = \tr(\mu^2) - \frac{1}{4}(\tr(\mu))^2$$
where  $\tr$ is the (rational) trace of an endomorphism as defined in \emph{\cite[p.182]{mumford1970abelian}}.
\end{proposition}

\paragraph{Irreducibility criterion}
We first recall the distinction between reducible and irreducible polarizations on an abelian surface.
\begin{definition}\emph{\cite[Satz 2]{Weilsatz}} \label{defn: reducpolar}A polarization $\theta\in\mathcal{P}(\A)$ is called \emph{reducible} (or \emph{decomposable}) if $\theta =
\cl(E_1 + E_2)$ for some elliptic curves $E_1$ and $E_2$ on $\A$.
\end{definition} The refined Humbert invariant provides a convenient way to detect the geometric type of principally polarized abelian surface, whether it is a Jacobian surface $(\mathcal{J(C)}, \theta_{\C})$ or a split surface $(E_1\times E_2,\theta_{E_1,E_2})$.
\begin{proposition}\emph{\cite[Proposition 6]{KaniMJ}} \label{irreducibilitycriteria}
Let $q_{(\A,\theta)}$ be the refined Humbert invariant  of the principally polarized abelian surface $(\A,\theta)$, then we have that
\begin{equation}
\theta \text{  is irreducible if and only if } q_{(\A,\theta)} \text{ does not represent } 1.
\end{equation}
\end{proposition}
If $\C$ is a  curve of genus $2$, its Jacobian $\mathcal{J(C)}$ has the canonical principal polarization $\theta_{\C}\in \NS(\mathcal{J(C)})$. In this case, we simply denote its refined Humbert invariant by $q_{\C}:=q_{(\mathcal{J(C)},\theta_{\C})}$.
\begin{definition}
    The \emph{degree map} is defined as a positive definite integral quadratic form (up to equivalence) denoted by $q_{E_1, E_2}$ such that $q_{E_1,E_2}(\varphi)= \deg(\varphi)$ for any $\varphi \in\Hom(E_1,E_2)$.
\end{definition}
The quadratic form $q_{E_1,E_2}$ is intrinsic on $\Hom(E_1,E_2)$; the phrase ``up to equivalence'' refers to its expression after choosing a $\Z$-basis.
For later use, let $\beta_d$ denote its associated bilinear form,
\[
\beta_d(\varphi,\varphi')=q_{E_1,E_2}(\varphi+\varphi')-q_{E_1,E_2}(\varphi)-q_{E_1,E_2}(\varphi').
\]

Now we focus on the case of products of two elliptic curves. Let $\A=E_1\times E_2$ and let $p_i:\A\rightarrow E_i$ be the natural projections. We write $\theta_i=p_i^*(0_{E_i})$ for $i=1,2$.

\begin{theorem} \emph{\cite[Proposition 23]{KaniMJ}}
Let $\A= E_1 \times E_2$ be a product surface. We have a group isomorphism 
\begin{align} \label{mainisomorphism}
    \mathcal{D}= \mathcal{D}_{E_1,E_2} : \Z \oplus \Z \oplus \Hom(E_1,E_2) \longrightarrow \NS(\A) \\\empty
    \mathcal{D}(a, b, \varphi) = (a-\deg (\varphi))\theta_1+(b-1)\theta_2 + \Gamma_{-\varphi}
\end{align}
where $\theta_i = p_i^{*}(0_{E_i})$, and $\Gamma_{-\varphi}$ is the graph of $-\varphi$. Then the rule $(a, b, \varphi)\rightarrow  \mathcal{D}(a, b, \varphi) \in \NS(\A)$ defines a group
isomorphism. Moreover, for two divisors $D_1= \mathcal{D}(a,b,\varphi)$ and $D_2= \mathcal{D}(a',b',\varphi')$ in $\NS(\A)$, the intersection number of the divisors is given by 
\begin{equation}\label{intersectionoftwodivisors}
    (D_1 \cdot D_2) = ab'+ a'b - \beta_d(\varphi,\varphi')
\end{equation}
where $\beta_d$  is the bilinear form associated to the $q_{E_1,E_2}$ on $\Hom(E_1,E_2)$. Thus
$$(\mathcal{D}(a, b, \varphi)\cdot\mathcal{D}(a, b, \varphi)) = 2(ab - \deg(\varphi)), \quad (\mathcal{D}(a, b, \varphi)\cdot(x\theta_1 + y\theta_2)) = bx + ay.$$
\end{theorem} 
If $\A=E_1 \times E_2$ is a product surface, where $E_1\sim E_2$ are isogenous, and elliptic curves $E_i$ for $i\in \{1,2\}$ have extra endomorphisms, then via the isomorphism $\mathcal{D}$ above, we can connect the intersection form $q_{\A}$ to the degree map $q_{E_1, E_2}$ on $\Hom(E_1, E_2)$ as
\[
q_{\A}(\mathcal{D}(x,y,\varphi))=xy-q_{E_1,E_2}(\varphi).
\]
Note that $q_{\A}$ is an indefinite integral quadratic form in $6$ variables in the case of our interest.

We can characterize principal polarizations in the following simple description.
\begin{proposition} \emph{ \cite[Corollary 25]{KaniMJ} }\label{polarizationcondition}
Let $\A=E_1 \times E_2$ be a product surface. Let $D = \mathcal{D}(a,b,\varphi) \in \NS(\A)$. Then $D \in \mathcal{P}(\A)$ if and only if $a > 0$ and $ab - \deg(\varphi) = 1$. Thus, every principal polarization of $\A$ has the form $\mathcal{D}(n_1,n_2,\varphi)$ with $\varphi\in \Hom(E_1,E_2)$ and $n_1,n_2 >0$ with $n_1 n_2- \deg(\varphi)=1$.
\end{proposition}

For elliptic curves $E/K$ and $E'/K$, let $\theta_E$ and $\theta_{E'}$ denote their canonical principal polarizations. We write $\theta_{\A}=\theta_E\otimes\theta_{E'}$ for the product principal polarization on $\A=E\times E'$.
For a product surface with the reducible polarization, the refined Humbert invariant is expressed in terms of the degree quadratic form on $\Hom(E, E')$ as follows.

\begin{lemma} \emph{\cite [Lemma 21]{KaniESCII}}
Let $E/K$ and $E'/K$ be elliptic curves over an arbitrary field $K$, and let $\A = E\times E'$ be the product surface with the product polarization $\theta_{\A} =\theta_E \otimes \theta_{E'}$. For $a,b\in \Z$, and $\varphi \in \Hom(E,E')$,
\begin{equation}\label{RHI=x2+4}
q_{(\A,\theta_{\A})}(\mathcal{D}(a, b, \varphi) + \Z\theta_{\A}) =(a-b)^2+4q_{E,E'}(\varphi),
\end{equation}
$q_{E,E'}$ denotes the degree form on $\Hom(E, E')$.
\end{lemma}
We state some useful facts used in our computations.
\begin{corollary}\emph{\cite[Lemma 28]{KaniESCI}}
The determinant of the Néron-Severi group of $E\times E'$ with respect to the intersection form is given by
$$ \det(\NS(E \times E')) = (-1)^{\rho -1} \det(\Hom(E,E'), \beta_d),$$
where $\rho =\rank(\NS(E\times E')) = \rank(\Hom(E, E'))+2$.
\end{corollary}
\begin{proposition}\emph{\cite[Lemma 30]{KaniESCI}}
Let $\rho = \rank(\NS(\A))$. Then the determinant of the quadratic module $(\NS(\A, \theta ), q_{(\A,\theta )})$ is related to that of the Néron–Severi group by the formula
$$\det(\NS(\A, \theta), q_{(\A,\theta )}) = \frac{1}{2} (-4)^{\rho -1}\det(\NS(\A), (.)).$$
\end{proposition}

We use three different integral quadratic forms: the intersection form $q_{\A}$, the refined Humbert invariant $q_{(\A,\theta)}$, and the degree map $q_{E_1, E_2}$ with their corresponding geometric objects N\'eron-Severi group $\NS(\A)$, the polarized N\'eron-Severi group $\NS(\A,\theta)$, and homomorphism module $\Hom(E_1, E_2)$, respectively. If $\A=E_1 \times E_2$ and $\End(E_i)\cong \mathcal{O}$ is a maximal order in a quaternion algebra $\Bp$ then the corresponding refined Humbert invariant $q_{(\A,\theta)}$ is a positive definite integral quadratic form in $5$ variables, which we refer to as a \emph{quintic} refined Humbert invariant. 

\subsubsection{Elliptic subcovers of abelian surfaces}
\label{subcover}
The curves with elliptic subcovers have a very long history\footnote{The history goes back to Legendre, Jacobi, Bolza, Humbert, Picard, Poincaré, and others. The classical method is given in terms of elliptic integrals in Legendre, Jacobi, Goursat \cite{bolza,brioschi,goursat,jacobi}, and a summary is given by Krazer \cite[Chapter XI]{krazer}.}. The different perspectives of these curves were studied by Hayashida and Nishi~\cite{hayashida1965existence}, Lange \cite{lange1,lange2}, Ibukiyama, Katsura, and Oort \cite{ibukiyama1986supersingular}, Kuhn \cite{kuhn}, Murabayashi \cite{murabayashi}, and Kani \cite{Kani94,kaniexistenceofellipticdiff,kaninumberofellipKani97}.

\begin{definition}
Let $\C$ be a smooth, projective curve of genus $2$ over a field $K$. An \emph{elliptic subcover} is defined as a
finite morphism $f: \C \rightarrow E$ to an elliptic curve $E/ K$ not factoring over a non-trivial isogeny of $E$.  If $f': \C \rightarrow E'$ is another elliptic subcover, then $f'$ is called \emph{equivalent} to $f$ if there exists an isomorphism $\varphi: E \Tilde{\rightarrow} E'$ such that $f' = \varphi \circ f$. 
\end{definition}
 We can suppose that $f$ is minimal\footnote{
Note that these morphisms are called \emph{maximal} by Serre \cite{serre}, and called \emph{optimal} by Kuhn \cite{kuhn}.}.
There is a bijection between the set of equivalence classes of
elliptic subcovers $\mathcal{E}(\C)$ of a genus $2$ curve $\C$ and the set of maximal elliptic subfields of the function field $F = K(\C)$ of $\C$. By \cite{bolza}, we know that $\mathcal{E}(\C)$ has either $0$ or $2$ or infinitely many elements.
For the case that $\C$ has infinitely many elliptic subcovers, it occurs when the Jacobian $\mathcal{J(C)}$ of $\C$ is $K$-isogenous to $E \times E$ for some elliptic curve $E/K$.
The set $\mathcal{E}(\C)$ was first studied by Kani \cite{Kani94} in the case $K = \overline{K}$ ; he proved that the refined Humbert invariant $q_{\C}$ provides a theoretical understanding of $\mathcal{E(\C)}$ \cite[Theorem 4.5]{Kani94}. Furthermore, Kani \cite{KaniESCI} later extended this result to any field as follows.
Recall that the Jacobian $\mathcal{J(C)}$ is called \emph{$(N,N)$-split} if and only if $\C$ has an
elliptic subcover of degree $N$. The criterion this yields is Theorem~\ref{SplittingTheorem},
stated in Section~\ref{sec:prelim} because our algorithms rest on it throughout.

\begin{definition}
    A \emph{presentation $(E, E', \psi)$ of degree $N$} of $\C/K$ for a given elliptic subcover $f : \C \rightarrow E$ of
degree $N$ is a triple containing $E$, another (isogenous) elliptic curve $E'/K$ and an
isomorphism $\psi: E[N] \rightarrow E'[N] $ which is an anti-isometry with respect to the Weil $e_N$
pairings. An invariant  $m_{\psi}$, called \emph{the isogeny defect} corresponding to $\psi$, is defined by
$$m_{\psi} := \min\{m \geq 1 : m\circ\psi = \varphi |_{E[N]} \textrm{ for some } \varphi \in \Hom(E, E')\}.$$
\end{definition}

Recall that $q_{E,E'}$ denotes the degree map on $\Hom(E,E')$, defined by $q_{E,E'}(\varphi)=\deg(\varphi)$ for $\varphi\in\Hom(E,E')$, and let $\det(q_{E,E'})$ denote its determinant.

\begin{theorem}\emph{\cite[Theorem 3]{KaniESCI}}
If $\C/K$ has a presentation $(E, E',\psi)$ of degree $N$ with $\charec(K) \nmid N$ and isogeny
defect $m = m_\psi$, and if $r = \rank(\Hom(E, E'
))\geq 1$, then the refined Humbert invariant $q_{\C}$ is
a positive definite quadratic form of rank $r+1$ which satisfies properties
\begin{itemize}
    \item [(i)]$\det(q_{\C}) = 2^{2r+1}m^2\det(q_{E,E'})$.
    \item [(ii)] $q_{\C}$ primitively represents $N^2$.
    \item [(iii)] $q_{\C}(x_1,\dots,x_{r+1}) \equiv 0, 1 \pmod  4$, \textrm{for all} $x_1,\dots,x_{r+1}\in \Z$.
    \item [(iv)] $q_{\C}(x_1,\dots,x_{r+1}) \neq1$ \textrm{for all } $x_1,\dots,x_{r+1}\in \Z$.
\end{itemize}
\end{theorem}

\begin{theorem}\emph{\cite[Theorem 4]{KaniESCI}}
If $(E, E'
,\psi)$ is a presentation of degree $N$ of a curve $\C/K$ of genus $2$ with $\charec(K)\nmid N$, then $m_\psi = 1$ if and only if $\mathcal{J(C)} \cong E\times E'$.
\end{theorem} 

\begin{corollary} \emph{\cite[Corollary 5]{KaniESCI}}
If $\C/K$ has an elliptic subcover $f: \C \rightarrow E$ with $\charec(K)\nmid \deg(f)$, if $\rho = \rank(\NS(\mathcal{J(C)}))\geq 3$, and if $\det(q_{\C})/2^{2\rho-3}$ is square-free, then $\mathcal{J(C)} \cong E \times E'$ for some elliptic curve $E'/K$.
\end{corollary}

\subsubsection{Automorphism groups}\label{sec:Auto}
The refined Humbert invariant $q_{\C}$ of a genus $2$ curve $\C$ has a close relation to the automorphism group of the curve $\C$, see \cite{kir2023refined,KaniCurvesAbelianSurfaces}. Kani~\cite{KaniCurvesAbelianSurfaces} showed that the automorphism group $\Aut(\C)$ of a genus $2$ curve $\C$ can be determined from its refined Humbert invariant $q_{\C}$ as follows.
\begin{proposition}\emph{\cite[Proposition 26]{KaniCurvesAbelianSurfaces}} \label{auto-rhi-relation}
Let $G$ be a finite group, and let
 $\iota(G) := |\{g \in G : \mathrm{ord}(g) = 2\}|$
denote the number of involutions of $G$. For any integer $n$, let
$r_n(q_{\C}) = |R_n(q_{\C})|$, where $R_n(q_{\C}) = \{ D \in \NS(\mathcal{J(C)},\theta_{\C}) : q_{\C}(D) = n\}$. If $\C/K$ is a  curve of genus $2$ , then
$\iota(\Aut(\C))-1 = r_4(q_{\C})$, the difference being non-negative because the hyperelliptic involution always lies in $\Aut(\C)$. 
Thus, if $r_4(q_{\C}) > 0$, then $\mathcal{J(C)} \sim E_1 \times E_2$ for some elliptic curves $E_i/K$, $i=1,2$. Moreover, if
$r_4(q_{\C}) \geq 3$, then $E_1 \sim E_2$.  
\end{proposition} 

The list of possibilities for $\Aut(\C)$ for any genus $2$ curve is given in the first row of Table~\ref{Tab:Autr4}, see~\cite[Theorem 2]{shaska2004elliptic}.  The number $r_4(q_\C)$ is given in the second row of Table~\ref{Tab:Autr4}.

\begin{table}[H]
    \centering
    \caption{Automorphism groups of a curve $\C$ of genus $2$ with $r_4(q_\C)$}
    \label{Tab:Autr4}
    \begin{tabular}{|c|c|c|c|c|c|c|c|}
    \hline
        $\Aut(\C)$ 
        & $\textrm{C}_2$ 
        & $\textrm{C}_{10}$ 
        & $\textrm{C}_2\times\textrm{C}_2$
        & $\textrm{D}_4$
        & $\textrm{D}_6$
        & $\textrm{C}_3\rtimes\textrm{D}_4$
        & $\GL_2(\mathbb{F}_3)$ \\ \hline
        $r_4(q_{\C})$
        & 0 & 0 & 2 & 4 & 6 & 8 & 12 \\ \hline
    \end{tabular}
\end{table}

\subsection{Real Multiplication, Humbert Surfaces, and Generalized Humbert Sets} \label{RM-Humbert}
Humbert \cite{Humbert} defined an integer invariant for abelian surfaces $\A/ \mathbb{C}$ with  $\End(\A)\supsetneq\Z$, known as the \emph{Humbert invariant}, giving rise to Humbert surfaces $\mathcal{H}_N$. 
The refined Humbert invariant $q_{(\A,\theta)}$, Definition~\ref{RHIdefn}, is an algebraic refinement of this classical invariant and gives an algebraic description of the Humbert surfaces $\mathcal{H}_N$. 
Let $\mathfrak{A}_2(K)$ denote the moduli space of principally polarized abelian surfaces over $K$.

\begin{definition}[Classical Humbert surface] \emph{\cite[Definition~3.7]{Kani94}}
 For a positive integer $N$, the \emph{Humbert surface of discriminant $N$} is
$\mathcal{H}_{N}:=\{\langle \A,\theta\rangle\in\mathfrak{A}_2(K):q_{(\A,\theta)}\text{ primitively represents }N\}$.
This is precisely the classical Humbert surface of discriminant (or invariant) $N$.
\end{definition}
This formulation can be deduced from \cite[\S2, Equation 8]{KaniGeneralizedHumbert}. Therefore, for a positive non-square discriminant $D$, the condition that $q_{(\A,\theta)}$ primitively represents $D$ is equivalent to the existence of an optimal embedding $\mathcal{O}_D\hookrightarrow \End(\A)$ with image contained in $\End_{\theta}(\A)$. Hence, $\mathcal{H}_D$ is precisely the locus of principally polarized abelian surfaces admitting  real multiplication by $\mathcal{O}_D$.

This description of Humbert surfaces in terms of primitive representations by the refined Humbert invariant extends naturally to positive definite quadratic forms in several variables.
Using the refined Humbert invariant, Kani \cite{KaniMJ} introduced subsets of the moduli
space $\mathfrak{A}_2$ of principally polarized abelian surfaces $(\A,\theta)$, called the \textit{generalized Humbert scheme/set}  $\mathcal{H}(q)$, indexed by positive definite integral quadratic forms $q$ in $n \geq 1$ variables. If $n= 1$, they are exactly the Humbert surfaces, and $\mathcal{H}(q)$ helps to determine the components of the intersection of Humbert surfaces. Every  $\mathcal{H}(q)$ is contained in any Humbert surface $\mathcal{H}_N$ such that $q$ primitively represents $N$ \cite[Proposition 5.2]{SubcoversGenus2}.

\begin{theorem}\label{RMtheorem}Let $(\A,\theta)$ be a principally polarized abelian surface over an algebraically closed field $K$, and let $D>0$ be a fundamental discriminant. Then $q_{(\A,\theta)}$ primitively represents $D$ if and only if there exists an embedding $\xi:\mathcal{O}_D\hookrightarrow\End(\A)$ such that $\xi(\mathcal{O}_D)\subseteq\End_{\theta}(\A)$. \end{theorem}
This follows from \cite[Proposition~14 and Remark~16]{KaniESCI}.

\begin{remark}
The formulation of Humbert surfaces used by Kani differs from that adopted by Kumar mainly in the choice of moduli-theoretic description. Kani defines the Humbert surface $\mathcal{H}_D\subset \mathfrak{A}_2$ intrinsically in terms of a principally polarized abelian surface $(\A,\theta)$, requiring the associated Humbert invariant quadratic form $q_{(\A,\theta)}$ to represent $D$ primitively. Kumar instead emphasizes the equivalent real multiplication interpretation, viewing $\mathcal{H}_D$ as the image in $\mathfrak{A}_2$ of the Hilbert modular surface parameterizing principally polarized abelian surfaces equipped with a compatible action of the quadratic order of discriminant $D$. In the square discriminant case, Kumar further restricts attention to the Jacobian locus to relate $\mathcal{H}_{N^2}$ to genus $2$ curves admitting degree-$N$ elliptic subcovers. Thus, Kani's definition is intrinsic to the principally polarized abelian surface and applies independently of whether the polarization is reducible or irreducible, whereas Kumar's genus $2$ interpretation concerns the intersection of the Humbert surface with the Jacobian locus.
\end{remark}

\section{Proof of Theorem~\ref{thm:local-automatic}}\label{app:proof}

\begin{proof}[Theorem~\ref{thm:local-automatic}]
By the preceding paragraph, we may take $q=q_5$ and must show: for every
odd prime $\nu$ and $t\in\Z_\nu$ there is
$\alpha\in\Z_\nu^5\setminus\nu\Z_\nu^5$ with $q_5(\alpha)=t$;
and for $\nu=2$, $p\equiv3\pmod4$, such $\alpha$ exists exactly when
$t\equiv0,1\pmod4$. We use two facts from the classification of quadratic
forms over $\Z_\nu$ in \cite[Ch.~5]{Kitaoka}, where the matrix $G_q$
appears as the matrix $(B(e_i,e_j))$ of the bilinear form $B$ with
$q(y)=B(y,y)$, and a form $q$ in $n$ variables is called \emph{unimodular}
over $\Z_\nu$ when $\det G_q\in\Z_\nu^{\times}$. Two forms are
\emph{equivalent over $\Z_\nu$} if one is obtained from the other by
a substitution $y=Uy'$ with $U\in \GL_n(\Z_\nu)$; this multiplies
$\det G_q$ by $(\det U)^2$ and maps
$\Z_\nu^n\setminus\nu\Z_\nu^n$ onto itself.

\emph{(i) Odd $\nu$.} Put $q_3(x_0,x_1,x_2):=q_5(x_0,x_1,x_2,0,0)
=x_0^2+4x_1^2+4x_2^2$. Then $G_{q_3}=\operatorname{diag}(1,4,4)$ has
determinant $16$, so $q_3$ is unimodular over $\Z_\nu$ for every odd
$\nu$, including $\nu=p$. By \cite[Theorem~5.2.4 and the proof of
Corollary~5.2.2]{Kitaoka}, a unimodular form in three variables over
$\Z_\nu$, $\nu$ odd, whose matrix has determinant $d$ is equivalent
over $\Z_\nu$ to $2y_1y_2-d\,y_3^2$; so there is
$U\in \GL_3(\Z_\nu)$ with $q_3(Uy)=2y_1y_2-16y_3^2$. For
$t\in\Z_\nu$ take $y=(1,\tfrac t2,0)$, which lies in
$\Z_\nu^3$ because $2$ is a unit and outside $\nu\Z_\nu^3$
because its first coordinate is $1$; then $q_3(Uy)=t$ and
$Uy\notin\nu\Z_\nu^3$, so $\alpha:=(Uy,0,0)\in\Z_\nu^5
\setminus\nu\Z_\nu^5$ satisfies $q_5(\alpha)=t$.

\emph{(ii) $\nu=2$, $p\equiv3\pmod4$.} Put
\[
q_4(y_1,y_2,y_3,y_4):=2y_1^2+2y_1y_2+\tfrac{p+1}2y_2^2
+2y_3^2+2y_3y_4+\tfrac{p+1}2y_4^2 ,
\]
so that $q_5(x)=x_0^2+2\,q_4(x_1,x_3,x_2,x_4)$ identically.
The matrix $G_{q_4}$ is block diagonal with two blocks
$\bigl(\begin{smallmatrix}2&1\\1&(p+1)/2\end{smallmatrix}\bigr)$ of
determinant $p$, so $q_4$ is unimodular over $\Z_2$; and $(p+1)/2$ is
even because $p\equiv3\pmod4$, so $q_4$ takes only even values. By
\cite[Theorem~5.2.5]{Kitaoka}, a unimodular form in four variables over
$\Z_2$ taking only even values is equivalent over $\Z_2$
either to $2y_1y_2+2y_3y_4$ or to $2y_1y_2+2y_3^2+2y_3y_4+2y_4^2$. The
matrices of these two forms are
$\bigl(\begin{smallmatrix}0&1\\1&0\end{smallmatrix}\bigr)\oplus
\bigl(\begin{smallmatrix}0&1\\1&0\end{smallmatrix}\bigr)$ and
$\bigl(\begin{smallmatrix}0&1\\1&0\end{smallmatrix}\bigr)\oplus
\bigl(\begin{smallmatrix}2&1\\1&2\end{smallmatrix}\bigr)$, of determinants
$(-1)^2=1$ and $(-1)\cdot3=-3$. The second is impossible for $q_4$, since
$\det G_{q_4}=p^2$ and $p^2/(-3)\equiv5\pmod8$ is not the square of a
$2$-adic unit. Hence
$q_4(Uy)=2y_1y_2+2y_3y_4$ for some $U\in \GL_4(\Z_2)$, and the
substitution $x_0=y_0$, $(x_1,x_3,x_2,x_4)=Uy$ shows that $q_5$ is
equivalent over $\Z_2$ to
\[
\tilde q_5(y_0,y_1,y_2,y_3,y_4):=y_0^2+4y_1y_2+4y_3y_4 .
\]
Since the substitution maps vectors outside $2\Z_2^5$ to vectors
outside $2\Z_2^5$, it suffices to decide which $t$ are values of
$\tilde q_5$ at such vectors. If $t\equiv0\pmod4$ then
$t=\tilde q_5(0,1,t/4,0,0)$, and if $t\equiv1\pmod4$ then
$t=\tilde q_5(1,1,(t-1)/4,0,0)$, with $t/4$, $(t-1)/4\in\Z_2$; both
vectors have second coordinate $1$, so lie outside $2\Z_2^5$.
Conversely $\tilde q_5(y)\equiv y_0^2\pmod4$ for every $y\in\Z_2^5$,
and the square of a $2$-adic integer is $\equiv0$ or $1\pmod4$; so
$\tilde q_5$ takes no value $\equiv2,3\pmod4$ at all. Hence $q_5$ takes the
value $t$ at a vector outside $2\Z_2^5$ if and only if
$t\equiv0,1\pmod4$.

The last assertions follow, as $t>0$ and every square is $\equiv0,1\pmod4$.
\end{proof}

\begin{remark}
Part (i) uses only the ternary form $q_3=x_0^2+4x_1^2+4x_2^2$ obtained from
$q_5$ by setting $x_3=x_4=0$; it is unimodular at every odd prime,
including $p$. The hypothesis $p\equiv3\pmod4$ enters only in (ii), to make
$q_4$ even-valued; for $p\equiv1\pmod4$ the cited classification of unimodular
$\Z_2$-lattices does not apply, and we make no claim either way. This costs us
nothing: the standing assumption $p\equiv11\pmod{12}$ of this paper implies
$p\equiv3\pmod4$, and is also what makes $\Bp=(-1,-p\mid\Q)$ and the base curve
$j=1728$ available in Section~\ref{designing}. Of $t=\ell^{2k}$, only the fact that it is a square is used.
\end{remark}

\section{The counts of Ibukiyama--Katsura--Oort}\label{app:iko}

For the reader's convenience, we reproduce the explicit values of $a_G(p)$ of
Theorem~\ref{Ibuyikama-auto-count}, taken from \cite[Theorem 3.3]{ibukiyama1986supersingular}, since
Table~\ref{auto-experiment} compares against them directly. Let $p\geq7$ and write
$\varepsilon_d=1-\legendre{-d}{p}$ for $d=1,2,3$, so that $\varepsilon_d\in\{0,2\}$. Then
\footnotesize\begin{align*}
a_{1}(p) &= \frac{(p-1)(p^2-35p+346)}{2880}-\frac{\varepsilon_1}{32}-\frac{\varepsilon_2}{8}-\frac{\varepsilon_3}{9}
   -\begin{cases}0 & p\equiv1,2,3\!\!\pmod 5,\\ \tfrac15 & p\equiv4\!\!\pmod 5,\end{cases}\\
a_{C_2}(p) &= \frac{(p-1)(p-17)}{48}+\frac{\varepsilon_1}{8}+\frac{\varepsilon_2}{2}+\frac{\varepsilon_3}{2},\\
a_{C_2\times C_2}(p) &= \frac{p-1}{8}-\frac{\varepsilon_1}{8}-\frac{\varepsilon_2}{4}-\frac{\varepsilon_3}{2},\qquad
a_{S_3}(p) = \frac{p-1}{6}-\frac{\varepsilon_2}{2}-\frac{\varepsilon_3}{3},\\
a_{D_6}(p) &= \frac{\varepsilon_3}{2}=\begin{cases}0 & p\equiv1\!\!\pmod 6,\\ 1 & p\equiv5\!\!\pmod 6,\end{cases}\qquad
a_{S_4}(p) = \frac{\varepsilon_2}{2}=\begin{cases}0 & p\equiv1,3\!\!\pmod 8,\\ 1 & p\equiv5,7\!\!\pmod 8,\end{cases}\\
a_{C_5}(p) &= \begin{cases}0 & p\equiv1,2,3\!\!\pmod 5,\\ 1 & p\equiv4\!\!\pmod 5.\end{cases}
\end{align*}
\textbf{The counts $\mathbf{h}(p)$ and $\mathbf{H}(p)$:} The two counts used throughout are also classical. The class number $\mathbf{h}$ of a maximal order
of $\Bp$, equivalently the number of supersingular elliptic curves over $\overline{\F}_p$ up to
isomorphism, is
\footnotesize\begin{equation}\label{eq:hcount}
  \mathbf{h}=\frac{p-1}{12}+\frac14\left(1-\legendre{-1}{p}\right)+\frac13\left(1-\legendre{-3}{p}\right),
\end{equation}
due to Eichler, Deuring and Igusa; and the number $\mathbf{H}$ of principal polarizations on
$E\times E$ up to automorphism, equivalently the number of principally polarized superspecial
abelian surfaces up to isomorphism, is
\footnotesize\begin{equation}\label{eq:Hcount}
\begin{aligned}
  \mathbf{H}={}&\frac{(p-1)(p+12)(p+23)}{2880}
   +\frac{2p+13}{96}\left(1-\legendre{-1}{p}\right)
   +\frac{p+11}{36}\left(1-\legendre{-3}{p}\right)\\
   &+\frac18\left(1-\legendre{-2}{p}\right)
   +\frac1{12}\left(1-\legendre{-3}{p}\right)\left(1-\legendre{-1}{p}\right)
   +\begin{cases}0 & p\equiv1,2,3\!\!\pmod5,\\[2pt] \tfrac45 & p\equiv4\!\!\pmod5,\end{cases}
\end{aligned}
\end{equation}
by \cite[Theorem 3.3]{KatsuraOort1987}; see also \cite[\S2.2]{edagaurish}, from which we take this
form of both statements. At $p=251$ these give $\mathbf{h}=22$ and $\mathbf{H}=6281$, so
$\frac{\mathbf{h}(\mathbf{h}+1)}{2}=253$ and $\mathbf{H}-\frac{\mathbf{h}(\mathbf{h}+1)}{2}=6028$,
reproducing the first three count columns of Table~\ref{auto-experiment}. Since
$\mathbf{h}=\frac{p}{12}+O(1)$ and $\mathbf{H}=\frac{p^{3}}{2880}+O(p^{2})$, the proportion of
principal polarizations that are reducible is
\begin{equation}\label{eq:redfrac}
  \frac{\mathbf{h}(\mathbf{h}+1)/2}{\mathbf{H}}=\frac{10}{p}+O(p^{-2}) .
\end{equation}
\noindent
\textbf{The automorphism counts:} Here $D_6$ denotes the dihedral group of order $12$, written $D_{12}$ in
\cite{ibukiyama1986supersingular}. The reduced automorphism group determines the full
automorphism group of the curve: the correspondence $\mathrm{RA}(\C)\mapsto\Aut(\C)$ used in
Table~\ref{auto-experiment} is
$1\mapsto C_2, C_2\mapsto C_2\times C_2,  C_2\times C_2\mapsto D_4, 
S_3\mapsto D_6, D_6\mapsto C_3\rtimes D_4, S_4\mapsto\GL_2(\F_3), C_5\mapsto C_{10}$,
which is compatible with the involution counts of Table~\ref{Tab:Autr4}.

As a check, take $p=251$, so $\legendre{-1}{p}=-1$, $\legendre{-2}{p}=1$, $\legendre{-3}{p}=-1$,
that is $(\varepsilon_1,\varepsilon_2,\varepsilon_3)=(2,0,2)$, and $p\equiv1\pmod5$. The formulas give
$a_1=4736$, $a_{C_2}=1220$, $a_{C_2\times C_2}=30$, $a_{S_3}=41$, $a_{D_6}=1$, $a_{S_4}=0$ and
$a_{C_5}=0$, which sum to $6028=\mathbf{H}(251)-\frac{\mathbf{h}(\mathbf{h}+1)}{2}$ and reproduce the
corresponding row of Table~\ref{auto-experiment}.

\end{document}